\documentclass[11pt]{amsart}

\usepackage[all]{xy}
\usepackage{mathrsfs}
\usepackage{amsmath, amsthm, amssymb, amscd, amsgen,amsxtra,amsfonts, amsbsy}
\allowdisplaybreaks
\usepackage{verbatim}

\usepackage{tikz}
\usepackage{tikz-cd}
\usepackage{color}
\definecolor{shadecolor}{gray}{0.875}
\definecolor{dblue}{rgb}{0,0,.6}
\usepackage[colorlinks=true, linkcolor=dblue, citecolor=dblue, filecolor = dblue, menucolor = dblue, urlcolor = dblue, pagebackref]{hyperref}
\usepackage{mathtools}
\usepackage{extarrows}
\usepackage{ifthen}

\usepackage{graphicx}
\usepackage{subcaption}

\newcommand{\mathds}[1]{{\mathbb #1}}

\numberwithin{equation}{section}

\begin{document}
\theoremstyle{definition}
\newtheorem{Definition}{Definition}[section]
\newtheorem*{Definitionx}{Definition}
\newtheorem{Convention}{Definition}[section]
\newtheorem{Construction}{Construction}[section]
\newtheorem{Example}[Definition]{Example}
\newtheorem{Examples}[Definition]{Examples}
\newtheorem{Remark}[Definition]{Remark}
\newtheorem*{Remarkx}{Remark}
\newtheorem{Remarks}[Definition]{Remarks}
\newtheorem{Caution}[Definition]{Caution}
\newtheorem{Conjecture}[Definition]{Conjecture}
\newtheorem*{Conjecturex}{Conjecture}
\newtheorem{Question}[Definition]{Question}
\newtheorem{Hypothesis}[Definition]{Hypothesis}
\newtheorem*{Questionx}{Question}
\newtheorem*{Acknowledgements}{Acknowledgements}
\newtheorem*{Notation}{Notation}
\newtheorem*{Organization}{Organization}
\newtheorem*{Disclaimer}{Disclaimer}
\theoremstyle{plain}
\newtheorem{Theorem}[Definition]{Theorem}
\newtheorem*{Theoremx}{Theorem}

\newtheorem*{thmA}{Theorem A}
\newtheorem*{thmB}{Theorem B}

\newtheorem{Claim}[Definition]{Claim}
\newtheorem*{Claim*}{Claim}
\newtheorem{Proposition}[Definition]{Proposition}
\newtheorem*{Propositionx}{Proposition}
\newtheorem{Lemma}[Definition]{Lemma}
\newtheorem{Corollary}[Definition]{Corollary}
\newtheorem*{Corollaryx}{Corollary}
\newtheorem{Fact}[Definition]{Fact}
\newtheorem{Facts}[Definition]{Facts}
\newtheoremstyle{voiditstyle}{3pt}{3pt}{\itshape}{\parindent}%
{\bfseries}{.}{ }{\thmnote{#3}}%
\theoremstyle{voiditstyle}
\newtheorem*{VoidItalic}{}
\newtheoremstyle{voidromstyle}{3pt}{3pt}{\rm}{\parindent}%
{\bfseries}{.}{ }{\thmnote{#3}}%
\theoremstyle{voidromstyle}
\newtheorem*{VoidRoman}{}

\newenvironment{specialproof}[1][\proofname]{\noindent\textit{#1.} }{\qed\medskip}
\newcommand{\blowup}{\rule[-3mm]{0mm}{0mm}}
\newcommand{\cal}{\mathcal}
\newcommand{\Aff}{{\mathds{A}}}
\newcommand{\BB}{{\mathds{B}}}
\newcommand{\CC}{{\mathds{C}}}
\newcommand{\BC}{{\mathds{C}}}
\newcommand{\EE}{{\mathds{E}}}
\newcommand{\FF}{{\mathds{F}}}
\newcommand{\GG}{{\mathds{G}}}
\newcommand{\HH}{{\mathds{H}}}
\newcommand{\NN}{{\mathds{N}}}
\newcommand{\ZZ}{{\mathds{Z}}}
\newcommand{\PP}{{\mathds{P}}}
\newcommand{\QQ}{{\mathds{Q}}}
\newcommand{\BQ}{{\mathds{Q}}}
\newcommand{\RR}{{\mathds{R}}}
\newcommand{\BA}{{\mathds{A}}}
\newcommand{\Liea}{{\mathfrak a}}
\newcommand{\Lieb}{{\mathfrak b}}
\newcommand{\Lieg}{{\mathfrak g}}
\newcommand{\Liem}{{\mathfrak m}}
\newcommand{\ideala}{{\mathfrak a}}
\newcommand{\idealb}{{\mathfrak b}}
\newcommand{\idealg}{{\mathfrak g}}
\newcommand{\idealm}{{\mathfrak m}}
\newcommand{\idealp}{{\mathfrak p}}
\newcommand{\idealq}{{\mathfrak q}}
\newcommand{\idealI}{{\cal I}}
\newcommand{\lin}{\sim}
\newcommand{\num}{\equiv}
\newcommand{\dual}{\ast}
\newcommand{\iso}{\cong}
\newcommand{\homeo}{\approx}
\newcommand{\mm}{{\mathfrak m}}
\newcommand{\pp}{{\mathfrak p}}
\newcommand{\qq}{{\mathfrak q}}
\newcommand{\rr}{{\mathfrak r}}
\newcommand{\pP}{{\mathfrak P}}
\newcommand{\qQ}{{\mathfrak Q}}
\newcommand{\rR}{{\mathfrak R}}
\newcommand{\OO}{{\cal O}}
\newcommand{\CO}{{\cal O}}
\newcommand{\CN}{{\cal N}}
\newcommand{\numero}{{n$^{\rm o}\:$}}
\newcommand{\mf}[1]{\mathfrak{#1}}
\newcommand{\mc}[1]{\mathcal{#1}}
\newcommand{\into}{{\hookrightarrow}}
\newcommand{\onto}{{\twoheadrightarrow}}
\newcommand{\Spec}{{\rm Spec}\:}
\newcommand{\BigSpec}{{\rm\bf Spec}\:}
\newcommand{\Spf}{{\rm Spf}\:}
\newcommand{\Proj}{{\rm Proj}\:}
\newcommand{\Pic}{{\rm Pic }}
\newcommand{\Mov}{{\rm Mov }}
\newcommand{\Nef}{{\rm Nef }}
\newcommand{\MW}{{\rm MW }}
\newcommand{\Br}{{\rm Br}}
\newcommand{\NS}{{\rm NS}}
\newcommand{\Sym}{\operatorname{Sym}}
\newcommand{\Aut}{{\rm Aut}}
\newcommand{\Autp}{{\rm Aut}^p}
\newcommand{\ord}{{\rm ord}}
\newcommand{\pr}{\mathrm{pr}}
\newcommand{\coker}{{\rm coker}\,}
\newcommand{\divisor}{{\rm div}}
\newcommand{\Def}{{\rm Def}}
\newcommand{\rank}{\mathop{\mathrm{rank}}\nolimits}
\newcommand{\sm}{\mathop{\mathrm{sm}}\nolimits}
\newcommand{\Ext}{\mathop{\mathrm{Ext}}\nolimits}
\newcommand{\EXT}{\mathop{\mathscr{E}{\kern -2pt {xt}}}\nolimits}
\newcommand{\Hom}{\mathop{\mathrm{Hom}}\nolimits}
\newcommand{\HOM}{\mathop{\mathscr{H}{\kern -3pt {om}}}\nolimits}
\newcommand{\chari}{\mathop{\mathrm{char}}\nolimits}
\newcommand{\ch}{\mathop{\mathrm{ch}}\nolimits}
\newcommand{\CH}{\mathop{\mathrm{CH}}\nolimits}
\newcommand{\supp}{\mathop{\mathrm{supp}}\nolimits}
\newcommand{\codim}{\mathop{\mathrm{codim}}\nolimits}
\newcommand{\IMAGE}{\mathop{\mathrm{Im}}\nolimits}
\newcommand{\Span}{\mathop{\mathrm{Span}}\nolimits}
\newcommand{\DIV}{\mathop{\mathrm{div}}\nolimits}
\newcommand{\calA}{\mathscr{A}}
\newcommand{\calH}{\mathscr{H}}
\newcommand{\calL}{\mathscr{L}}
\newcommand{\calM}{\mathscr{M}}
\newcommand{\bcalM}{\overline{\mathscr{M}}}
\newcommand{\calN}{\mathscr{N}}
\newcommand{\calX}{\mathscr{X}}
\newcommand{\calK}{\mathscr{K}}
\newcommand{\calD}{\mathscr{D}}
\newcommand{\calY}{\mathscr{Y}}
\newcommand{\calC}{\mathscr{C}}
\newcommand{\calS}{\mathscr{S}}
\newcommand{\calE}{\mathscr{E}}
\newcommand{\calF}{\mathscr{F}}
\newcommand{\calG}{\mathscr{G}}
\newcommand{\CM}{\mathcal{M}}
\newcommand{\CK}{\mathcal{K}}
\newcommand{\CV}{\mathcal{V}}
\newcommand{\piet}{{\pi_1^{\rm \acute{e}t}}}
\newcommand{\Het}[1]{{H_{\rm \acute{e}t}^{{#1}}}}
\newcommand{\Hfl}[1]{{H_{\rm fl}^{{#1}}}}
\newcommand{\Hcris}[1]{{H_{\rm cris}^{{#1}}}}
\newcommand{\HdR}[1]{{H_{\rm dR}^{{#1}}}}
\newcommand{\hdR}[1]{{h_{\rm dR}^{{#1}}}}
\newcommand{\loc}{{\rm loc}}
\newcommand{\et}{{\rm \acute{e}t}}
\newcommand{\defin}[1]{{\bf #1}}

\renewcommand{\HH}{{\rm{H}}}

\ifthenelse{\equal{1}{1}}{
    \newcommand{\blue}[1]{{\color{blue}#1}}
    \newcommand{\green}[1]{{\color{green}#1}}
    \newcommand{\red}[1]{{\color{red}#1}}
    \newcommand{\purple}[1]{{\color{purple}#1}}
}{
    \newcommand{\blue}[1]{}
    \newcommand{\green}[1]{}
    \newcommand{\red}[1]{}
    \newcommand{\purple}[1]{}
}

\title{Isotrivial smooth curves on surfaces}

\author{Xi Chen}
\address{632 Central Academic Building, University of Alberta, Edmonton, Alberta T6G 2G1, Canada}
\email{xichen@math.ualberta.ca}

\author{Frank Gounelas}
\address{Mathematik, Universit\"at Bonn, Endenicher Allee 60, 53115 Bonn, Germany}
\email{gounelas@math.uni-bonn.de}

\date{\today}
\subjclass[2020]{14J28, 14H10, 14C20, 14J20}
\keywords{Isotrivial family of curves, moduli map, K3, hypersurface, DPC}

\begin{abstract}
    We prove that a smooth projective non-uniruled surface of Picard rank one, generated by an ample and base point free line bundle, cannot be covered by an isotrivial family of smooth curves in the primitive polarisation.
\end{abstract}

\maketitle
\tableofcontents

\section{Introduction}

In this paper, we study the existence of one-parameter isotrivial families of curves on a smooth projective surface $S$. It is easy to construct examples of such families, e.g., projective space or Fermat hypersurfaces - see Section \ref{sec:k3hyps} for a list of these and other examples in the K3 case. It is expected though, that ``most'' surfaces do not admit isotrivial families of curves, singular or not. 

The main purpose of this paper is to rule this out from happening with smooth curves in the primitive polarisation of a surface of Picard rank one, without any genericity assumptions on the surface.

\phantomsection\label{thm:A}
\begin{thmA}
    Let $S$ be a smooth projective surface such that $\NS(S) = \mathbb{Z} H$ for $H$ ample and base point free and suppose that $K_S+H$ is nef. Then the moduli morphism
    \[ \begin{tikzcd} {\mu\colon |H|_{\sm}} \arrow[r] & \mathcal{M}_g \end{tikzcd} \]
    is quasi-finite, i.e., has finite fibres.
\end{thmA}

We note that quasi-finiteness cannot in general be improved to finiteness, i.e., properness of $\mu$, even in the primitive setting above; see Remark \ref{rem:properness} for a discussion.

The proof actually implies something more general for families of surfaces which may vary in moduli (see Theorem \ref{K3ISOTRIVIALTHMNORMAL}), another special case of which we state as follows.

\begin{thmB}\label{thm:B}
    Let $S$ be a smooth projective surface with $K_S$ nef, and $H$ an ample and base point free divisor. Let $\mu \colon |H|_{\sm} \to \mathcal{M}_g$ be the moduli morphism. If there exists a complete curve $D \subset |H|$ on which $\mu$ is constant, then $D$ must contain a reducible or non-reduced curve.
\end{thmB}

The main applications we have in mind, and for which, to our knowledge, the above result was not known, are to K3 surfaces and to surfaces in $\PP^3$ of low degree. When the degree of the K3 or hypersurface in question is large enough, stability restriction theorems kick in and rule this out from happening (see Section \ref{sec:k3hyps}).

Note that it is easy to rule out smooth isotrivial families of curves over a compact base (i.e., having no singular fibres), since such a compact subvariety in the linear system must meet the ample discriminant divisor. It is also somewhat understood by the experts that the above is true for a very general K3 surface (see Theorems \ref{thm:moonen}, \ref{thm:totaro}, \ref{thm:cds}) and hypersurface in $\PP^3$, so the novel contributions are really
\begin{enumerate}
    \item that the more explicit statement that this is true for all surfaces of Picard rank one holds,
    \item and the method of proof.
\end{enumerate}
Whether there exists any K3 surface with no isotrivial families of (singular) curves remains open still (see Section \ref{sec:k3hyps}).

The existence of smooth isotrivial families in multiples of the primitive polarisation on surfaces of Picard rank one also remains open in general. If the self-intersection is large enough then it is known by stability restriction results (see Section \ref{sec:k3hyps}), but in general (or to give a more geometric argument) one must analyse  stable reductions of limits of generically smooth isotrivial families. We aim to address this in a future paper.

It is worth noting that the question in the opposite direction, namely that of maximal variation of a linear system, is much better understood. For example \cite{DuttaHuybrechts} observed that the Matsushita conjecture for Lagrangian fibrations, recently settled by Bakker \cite{bakker}, implies that for any pair of a K3 surface and an ample base point free linear system $H$, the induced rational map \[ \begin{tikzcd} {|H|} \arrow[r, dashed] & \mathcal{M}_g \end{tikzcd} \] is generically finite onto its image. See also the recent preprints \cite{beauville2025maximal} and \cite{bricallipirola2026maximal} for related developments for other surfaces. 

As far as existence results go in the two main examples, from \cite{regenerationinfinite, maxmoduli}, for any genus $g\geq0$ and any K3 surface $X$, there exist infinitely many irreducible (potentially singular) curves $C\subset X$ of distinct class in the Picard group which vary maximally in moduli in their linear system (i.e., each varies in a different $g$-dimensional family). For hypersurfaces, an answer to this question is only known in the very general case. In that setting, Xu \cite{xu1994} showed that every irreducible curve on a very general surface of degree $d\geq5$ in $\PP^3$ has geometric genus at least $\tfrac{d(d-3)}{2}-2$, so that in particular no rational or low-genus curves occur, and in \cite{CFZ} they study the gaps in the geometric genera realised by such curves.

We also note that varieties whose generic hyperplane sections are projectively isomorphic (i.e., minimal variation, or that the moduli map is a contraction) have been studied and classified in \cite{beauville1990,mckernan1993, pardini1994}.

A few words on the proof of the above theorems, which proceeds by contradiction. First one takes a resolution and finite cover to make the family of curves trivial, and then a minimal resolution $Z$ of the resulting rational map to $S$ as follows
\[
\begin{tikzcd}[row sep=large, column sep=large]
    Z \arrow[d, "\psi"'] \arrow[r, "\phi"] & W \arrow[d, hook] \\
    C\times D \arrow[r, dashed, "f"] \arrow[dr, dashed] & S\times D \arrow[d, "\pi_1"] \\
    & S
\end{tikzcd}
\]
where $W$ denotes the image of $Z$ in $S\times D$. Since the trivialised family $C\times D\to D$ is smooth and $\psi$ is a sequence of blow-ups of points, $Z\to D$ is a semistable reduction of the family of curves: $Z$ is smooth and every fibre of $Z/D$ is reduced with normal crossings. 
The contradiction is then derived by comparing two Euler characteristic computations of twists of multiples of $K_{S\times D}+W$ by various general points of $C$ on both $W$ and on $Z$. The discrepancy between the two numbers comes about by contributions of curves in $Z$ which are contracted to $W$, creating singularities in the fibres of $W/D$, i.e., the components of the semistable reductions of the flat limits of the isotrivial family. The whole computation boils down to showing that in this setup, $C$ supports certain Brill--Noether special line bundles. Concretely, the crux of the reduction is to show that the divisor
\[
    \big((2g-1)q - s_1 - \dots - s_g\big) + m\,(q - s_{g+1}) \ \in\ \Pic^{g-1}(C)
\]
has sections for a sufficiently large positive integer $m$, where $q\in C$ is the special point over which the rational tail of the semistable reduction is attached and $s_1,\dots,s_{g+1}$ are $g+1$ general points of $C$. This follows by an intersection-theoretic computation with the theta divisor in the Jacobian of $C$ (see Section \ref{sec:thetadiv}).

\textbf{Conventions.}
A variety is always reduced but might be reducible and we work over $\CC$ throughout. A family of curves will be called isotrivial if any two general members (sometimes all, depending on context) are isomorphic.
We will often abuse notation and call a curve $C\subset X$ \textit{isotrivial} if it moves in an isotrivial family in $X$, or use expressions such as ``smooth isotrivial family of curves'' to refer to a smooth projective family of curves over an irreducible quasi-projective base all of whose closed geometric fibres are isomorphic.

\textbf{Acknowledgments} This paper benefited from many conversations with Yajnaseni Dutta and we would like to thank her. We would also like to thank A. Beauville for comments on a first draft. Anthropic's Fable LLM found a mistake in a first draft of this paper (see Example \ref{K3ISOTRIVIALEXAMPLEFABLE}).

\section{Background on isotrivial families of curves on surfaces}\label{sec:prelim}

In this section we collect some preliminary results and complementary remarks to \hyperref[thm:A]{Theorem~A}.

We start with the following easy fact, ruling out complete isotrivial families of smooth (or with mild singularities) curves.

\begin{Lemma}\label{lem:unramified}
    Let $S$ be a smooth projective surface over $\CC$. If there exist two smooth projective curves $C$
    and $D$ and a dominant morphism $f\colon Y = C\times D \to S$ such that $f\colon Y_t\to S$ is unramified for
    all fibres $Y_t\cong C$ of $Y/D$, then $c_2(S) = c_1(S) \cdot A \in \CH_\QQ^2(S)$ for some $A \in \Pic_\QQ(S)$.
\end{Lemma}
\begin{proof}
    If such $C, D$ and $f$ exist, then $f^*\Omega_S \to \Omega_{Y/D}$ is surjective and we have an exact sequence
    \[
        \begin{tikzcd}
            0 \ar{r} & L \ar{r} & f^*\Omega_S \ar{r} & \Omega_{Y/D} \ar{r} & 0
        \end{tikzcd}
    \]
    for a line bundle $L$ on $Y$. It follows that
    \[
        f^* c_2(S) = c_2(f^*\Omega_S) = (f^* c_1(K_S)) (\pi^* c_1(K_C))
    \]
    where $\pi$ is the projection $Y\to C$, since $c_2(f^*\Omega_S) = c_1(L)\,c_1(\Omega_{Y/D})$ with $c_1(L) = f^*c_1(K_S) - \pi^*c_1(K_C)$, and $(\pi^*c_1(K_C))^2 = 0$. Then
    \[
        (\deg f) c_2(S) = c_1(K_S) f_* (\pi^* c_1(K_C))
    \]
    as required.
\end{proof}

The following two lemmas provide general bounds and properties for families of curves on arbitrary smooth projective varieties, which we will use in a later section.

\begin{Lemma}\label{lem:H0neq0}
    Let $X$ be a smooth projective variety and let $F\colon C\times D\dashrightarrow X$ be a dominant rational map from the product of two smooth projective curves. Then for any $d\in D$ we have
    \[
        \HH^0(C, F_d^*T_X)\neq0.
    \]
    In particular, if $C\subset X$ is a smooth projective curve such that $\HH^0(C,T_X|_C)=0$, then $C$ cannot deform isotrivially in $X$.
\end{Lemma}
\begin{proof}
    We get a morphism $D\to \Hom(C,X)$, which induces a non-zero map $\CC\cong T_dD\to
        T_{[F_d]}\Hom(C,X)$. As the latter tangent space is $\HH^0(C, F_d^*T_X)$ we obtain the
    result.
\end{proof}

\begin{Lemma}\label{lem:fibredim}
    Let $X$ be a non-uniruled smooth projective variety of dimension $n$, and let $F\colon C\times T\dashrightarrow X$ be a dominant rational map from the product of a smooth projective curve $C$ and
    an irreducible quasi-projective variety $T$. Then $\dim T\leq n$.
\end{Lemma}
\begin{proof}
    From Mumford's rigidity theorem, as $F$ is dominant, no fibre of the projection to $T$ can get contracted via $F$. If $\dim T\geq n+1$, then for any
    $c_0\in C$ and $x_0\in X$ we can choose a curve $A=A_{c_0,x_0}\subset T$ (namely
    $A:=\pr_2(F^{-1}(x_0) \cap (\{c_0\}\times T))\subset T$) so that for every $a\in A$, the morphism
    $F_a\colon C\to X$ sends $F_a(c_0)=x_0$. This means that $\dim_{[F_a]}\Hom(C, X; c_0\mapsto x_0)\geq 1$, and so bend
    and break gives that there is a rational curve through $x_0$. Repeating this by varying $x_0$ among all the points
    of a fixed image $F_t(C)$ gives that $X$ must be uniruled, which is a contradiction.
\end{proof}

\section{Some examples}\label{sec:k3hyps}

\subsection{K3 surfaces}

As $h^1(\OO_S)=0$ and the discriminant locus of a linear system is an ample divisor, every isotrivial family of
smooth curves on a K3 surface must have singular flat limits (see alternatively Lemma \ref{lem:unramified}). One of the
original motivations for this paper was to understand these flat limits in the case of Picard rank one.

In general, the expectation here is the following, and could be attributed (as a question) to Schoen.

\begin{Conjecture}
    If $S$ is a very general K3 surface, then there can be no dominant rational map from the product of two
    smooth projective curves to $S$.
\end{Conjecture}

In other words, there is no irreducible curve in $S$ which moves isotrivially, singular or not. \hyperref[thm:A]{Theorem~A} deals with a more precise form of this conjecture in the smooth case.

Before mentioning some cases in which this conjecture is known, we begin by giving some examples where the conjecture can fail for special K3 surfaces.

\begin{Example}\label{ex:isotriv}
    \begin{enumerate}
        \item If $S=V(x_0^4+\ldots+x_3^4)$ is the Fermat quartic, then $x_0=tx_1$ is a 1-dimensional family of smooth
              hyperplane sections of $S$, i.e., plane curves of genus $3$, all of which have the same moduli, so
              the moduli map is constant.
        \item There exist various families of isotrivially elliptic K3 surfaces, so the moduli map is constant with
              1-dimensional fibre. Notably, the Weierstrass model $y^2=x^3+t^{12}-t$ gives an isotrivially elliptic K3 surface which has Picard rank 2; see Keum \cite{keum2016} and Kond\=o \cite{kondo1992}. See also
              \cite{moonen} Family no.\ 150 for a 9-dimensional family.
        \item\label{ex:isotrivkummer} If $S=\widetilde{A/i}$ is a Kummer surface, then any curve $C\subset A$ with $C^2>0$ (i.e., $g(C)\ge 2$) moves in a 2-dimensional
              family in $A$ by translation via the group law on $A$, so this induces a 2-dimensional dominant
              family
              \[\begin{tikzcd}
                      C\times A\arrow[d]\arrow[r, dashed] & S \\ A&
                  \end{tikzcd}\]
              and the induced moduli map is constant with 2-dimensional fibre. If $C^2 = 2n$, the generic translates $C_x$ and $i(C_x)$ meet in $2n$ points $y_1, i(y_1), \dots, y_n, i(y_n)$ away from the fixed points, which project to $n$ nodes on the image of $C_x$ in $S$, so all these curves are singular in $S$.
        \item There exist non-Kummer quotients of abelian surfaces whose resolution is K3. Any curve in the
              abelian surface again moves in a 2-dimensional family. These have been classified by Fujiki and
              others (see Garbagnati \cite{garbagnati}) and have high Picard rank.
        \item Paranjape \cite{paranjape} constructs an isotrivial family of genus 5 curves covering a non-Kummer K3 surface of Picard rank 16 (the desingularization of the double cover of $\PP^2$ branched along 6 general lines).
    \end{enumerate}
\end{Example}

An interesting feature of Example \ref{ex:isotriv}.\eqref{ex:isotrivkummer} is that the base of the covering family is
two-dimensional. Lemma \ref{lem:fibredim} extends \cite[Corollary 3.2]{DuttaHuybrechts} to the whole linear system as opposed to just a general member, and says
that there can be no isotrivial families with dimension of the base larger than two. 
\begin{Remark}
    A reasonable criterion to distinguish which K3 surfaces can be covered by a one- resp.two-dimensional family of curves eludes us and would be interesting.
\end{Remark}

\begin{Theorem}[Moonen \cite{moonen}]\label{thm:moonen}
    Assume that $\Pic(S)\cong\ZZ$ and $\operatorname{End}_{\QQ}(T)=\QQ$ where $T=T(S)$ is the transcendental lattice of
    $S$,  and that there exists $C\times D\dashrightarrow S$ a dominant rational map from the product of two smooth projective curves.  Then $g(C), g(D) \geq 512$. 
\end{Theorem}
In other words, there can be no irreducible isotrivial curves on $S$ of geometric genus $<512$.
The condition on the endomorphism algebra is a generic one.

In this paper we are interested in results applying to surfaces of Picard rank one, and not just very general, and here for high genus there are the following two options.

\begin{Theorem}[Totaro, \cite{totaro}]\label{thm:totaro}
    Let $S$ be a K3 surface with $\Pic S\cong \ZZ H$ and assume that $H^2=20$ or $H^2\geq24$. Then for all $n\geq1$, there are no smooth curves in $|nH|$ which deform isotrivially.
\end{Theorem}
\begin{proof}
    From \cite[Theorem 3.3]{totaro}, Bott vanishing holds for
    $(S,C)$, so that in particular $\HH^1(S, T_S(C))=0$. Taking cohomology of the long exact
    sequence \[ \begin{tikzcd} 0 \arrow[r] & T_S(-C) \arrow[r] & T_S \arrow[r] & {T_S|_C} \arrow[r] & 0 \end{tikzcd} \] and using Serre duality and $T_S\cong\Omega^1_S$, we
    obtain that $\HH^0(T_S|_C)=0$. In particular, from Lemma \ref{lem:H0neq0}, $C$ cannot
    deform isotrivially in its linear system.
\end{proof}

More precise results are known for lower genera from \cite{knutsen-globalsections} and \cite{dedieu-sernesi}.

\begin{Remark}\label{rem:stable}
    We note that as $\Omega^1_S$ is a $\mu$-stable vector bundle with respect to any ample line bundle, various restriction theorems of Bogomolov  and Hein \cite{ghein} (see \cite[Theorem 2.1, Remark 2.7.(i)]{DuttaHuybrechts}) imply that $\Omega^1_S|_C$ often remains stable, for example in the case where $C$ is any smooth curve in its linear system with $C^2\geq 49$ (which is a slightly worse bound than that in Bott vanishing above). As the bundle in question has degree 0, it follows that $\HH^0(C,T_S|_C)=0$ and Lemma \ref{lem:H0neq0} implies a similar result for this range.
\end{Remark}

Alternatively one can use the following, more general result which does not require any assumptions on the Picard rank (or genericity).

\begin{Theorem}[{Ciliberto--Dedieu--Sernesi, \cite[Proposition 8.6]{CDS}}]\label{thm:cds}
    Let $S$ be a K3 surface and $C\subset S$ a curve with $C^2\geq20$ and $\mathrm{Cliff}(C)>2$. Then there are only finitely many $C'\in|\OO_S(C)|$ so that $C'\cong C$.
\end{Theorem}

The condition on the Clifford index is satisfied by all smooth curves $C$ if the Picard rank is one and $C^2\geq12$ from \cite[Theorem 1.3]{knutsenk-th}.  

All in all, the above results imply that for the case of a K3 surface $S$ of Picard group generated by $H$, there can be no smooth isotrivial families of curves in $|nH|$ unless
\begin{itemize}
    \item $H^2=2$ and $1\leq n\leq3$,
    \item $H^2=4$ and $1\leq n\leq2$,
    \item $6\leq H^2\leq 18$ and $n=1$.
\end{itemize}
\hyperref[thm:A]{Theorem~A} deals with all the cases $n=1$ of the above, and in particular to rule out smooth isotrivial families of curves in Picard rank one K3 surfaces, all that remains are the three cases $(H^2,n)=(2,2), (2,3), (4,2)$. These will be dealt with in a sequel to this paper with Y. Dutta.

\subsection{Very general surfaces in $\PP^3$}

We begin with the following theorem of Schoen, which rules out isotrivial families of curves, singular or not, in very general hypersurfaces.

\begin{Theorem}[Schoen, \cite{Schoen}]
    The very general surface $S\subset\PP^3$ of degree $d\geq5$ is not dominated by a product of curves.
\end{Theorem}

Using a simple dimension count argument, one can also rule out smooth isotrivial families of curves in the primitive (hyperplane) class if $d\geq4$, but we do not include this argument here as our focus is on Picard rank one surfaces.

Similar arguments to Remark \ref{rem:stable} apply to hypersurfaces in $\PP^3$ of degree $d\geq5$, and give that if $S$ has Picard rank one and $C$ is a smooth curve in its linear system with $C^2\gg0$, then $T_S|_C$ is stable and in particular has no sections, so the result follows similarly from Lemma \ref{lem:H0neq0}. See \cite{ghein} for precise numerics.

We remark also that it is possible to show that there exist hypersurfaces $S$ of Picard rank one for which the moduli map

\[
    \begin{tikzcd} {\mu\colon |\OO(1)|_{\sm}} \arrow[r] & \mathcal{M}_g \end{tikzcd}
\]

is quasi-finite but of degree at least two. We expect this should not happen for the very general hypersurface, namely that the map is actually one-to-one in this case. In upcoming work of the second author with D. Faro the follow-up question of if and where $\mu$ is unramified is studied, i.e., whether the group $\HH^0(S, T_S|_C)=0$ for smooth curves.

\section{Auxiliary Results}\label{sec:auxiliary}
\subsection{A Result on Theta Divisors}\label{sec:thetadiv}

Let $C$ be a smooth projective curve of genus $g\geq2$. For a divisor $D\in\Pic^{g-1}(C)$, a fixed point $p\in C$ and an $m>0$, we will be interested in this section in the cardinality of the set
\[
    \Sigma_m = \big\{q\in C: h^0(D+m(p-q)) > 0\big\}
\]
as $m$ goes to infinity.

Let $\alpha_p\colon C\to J(C)$ be the Abel--Jacobi map $\alpha_p(q) = p-q$
\[\Theta = \{G\in J(C): h^0(G + (g-1)p) > 0\}\]
is the theta divisor
and $[m]\colon J(C)\to J(C)$ is the multiplication map $[m](G) = mG$, then $q\in \Sigma_m$ if and only if
\[[m] \circ \alpha_p(q)\in \Theta + ((g-1) p - D) = \Theta + \Delta\]
where $\Delta = (g-1)p - D\in J(C)$
and $\Theta + \Delta$ is the translation of $\Theta$ by $\Delta$.
In other words,
\[
    \Sigma_m = ([m] \circ \alpha_p)^{-1} (\Theta + \Delta).
\]
The following proposition will prove that
\begin{itemize}
    \item for all $m\geq g$ and a $D\in\Pic^{g-1}(C)$, the curve $[m](\alpha_p(C))$ is not contained in $\Theta + \Delta$,
    \item if $D\in\Pic^{g-1}(C)$ is very general, then in fact as $m\to\infty$,
          \[[m](\alpha_p(C)).(\Theta + \Delta)\to\infty.\]
\end{itemize}

\begin{Proposition}\label{K3ISOTRIVIALPROPJACOBIAN}
    Let $C$ be a smooth projective curve of genus $g\ge 2$, let $p$ be a point on $C$ and let $D\in \Pic(C)$ be a divisor on $C$ of degree $d$.
    \begin{itemize}
        \item For every $m\ge g$ and a general point $q\in C$,
              \begin{equation}\label{K3ISOTRIVIALTHMGLOBALE034}
                  h^0(D+m(p-q)) = \max(0, d - g + 1).
              \end{equation}
        \item If $d = g-1$ and $D\in \Pic^{g-1}(C)$ is very general,
              then
              \begin{equation}\label{K3ISOTRIVIALTHMGLOBALE035}
                  \lim_{m\to\infty} \Big |\big\{q\in C: h^0(D+m(p-q)) > 0\big\}\Big| = \infty
              \end{equation}
    \end{itemize}
\end{Proposition}
\begin{proof}[Proof of Proposition \ref{K3ISOTRIVIALPROPJACOBIAN}]
    For \eqref{K3ISOTRIVIALTHMGLOBALE034}, we prove the result first for $d = g-1$. In this case, for $m\geq g$
    \[
        h^0(D + mp) = m.
    \]

    Suppose \eqref{K3ISOTRIVIALTHMGLOBALE034} fails for $d=g-1$. Equivalently, the line bundle $L:=D+mp$ has a nonzero section vanishing to order $\ge m$ at a general point of $C$. Choose a small disk $\Delta \subset C$ with a local coordinate $t$, over which the dimension $h^0(L-mt)$ is constant. The subspaces $H^0(L-mt)\subset H^0(L)$ then form a holomorphic subbundle, so we may pick a holomorphic family $s_t \in H^0(L)$ of nonzero sections such that $\operatorname{mult}_t s_t=n$ for some constant $n\ge m$. The parametric derivatives
    \[
        s_t^{(k)} := \frac{d^k}{dt^k}\,s_t \in H^0(L), \qquad k=0,1,\dots,n,
    \]
    again lie in the fixed vector space $H^0(L)$. Expanding locally around the point $t$, we can write $s_t(x)=\sum_{j\ge n}c_j(t)(x-t)^j$ where $c_n(t)\ne 0$. The leading term contributes 
    \[ 
        (-1)^k n(n-1)\cdots(n-k+1)\,c_n(t)\,(x-t)^{n-k}
    \]
    to $s_t^{(k)}$ and so $\operatorname{mult}_t s_t^{(k)}=n-k$. Hence $s_t^{(0)},\dots,s_t^{(n)}$ have pairwise distinct vanishing orders at the point $t$ and are linearly independent, giving $h^0(L)\ge n+1$. But for $m\ge g$ one has $\deg L=g-1+m>2g-2$, so $h^0(L)=m$ by Riemann--Roch, and $m<n+1$ is a contradiction. This proves \eqref{K3ISOTRIVIALTHMGLOBALE034} for $d=g-1$.

    If $d>g-1$ then we may write $D=D_1+D_2$ where $\deg(D_1)=g-1$ and $D_2$ is effective. From the $d=g-1$ case and Riemann--Roch, we conclude that $h^1(D+m(p-q))=0$ for a general point $q\in C$. The case $d<g-1$ is similar, writing $D_1=D+D_2$ where $\deg(D_1)=g-1$ and $D_2$ is effective.

    For $\Theta, \Delta$ as above the proof, note that by \eqref{K3ISOTRIVIALTHMGLOBALE034}, $[m]\circ \alpha_p(C)$ and $\Theta + \Delta$ meet properly for all $m\ge g$.

    Since $[m]^*\colon H^2(J(C), \ZZ)\to H^2(J(C), \ZZ)$ is given by $[m]^*(\alpha) = m^2 \alpha$, we have
    $c_1([m]^* \Theta) = m^2 c_1(\Theta)$ and hence
    \[
        \deg \alpha_p^* [m]^* \Theta = \alpha_{p,*} C . [m]^* \Theta = m^2 \alpha_{p,*} C . \Theta
    \]
    And since $\Theta$ is ample, we have
    \begin{equation}\label{K3ISOTRIVIALTHMGLOBALE042}
        \lim_{m\to\infty} \deg \alpha_p^* [m]^* \Theta = \infty.
    \end{equation}
    The same holds with $\Theta$ replaced by $\Theta + \Delta$ in \eqref{K3ISOTRIVIALTHMGLOBALE042}.

    Since $D\in \Pic^{g-1}(C)$ is very general, $\Delta\in J(C)$ is very general. So
    $\alpha_p^* [m]^* (\Theta + \Delta)$ is reduced on $C$ for all $m$ by Kleiman's generalised Bertini's Theorem (cf. \cite[Theorem 10.8, p. 273]{Hartshorne}). Therefore,
    \[
        \lim_{m\to\infty} |\Sigma_m| = \lim_{m\to\infty} |\alpha_p^{-1}\circ [m]^{-1} (\Theta+\Delta)| = \infty.\qedhere
    \]
\end{proof}

\subsection{On families of stable maps}

The setup will be as follows

\begin{Hypothesis}\label{K3ISOTRIVIALTHMHYP}
    Let
    \[
        \begin{tikzcd}
            f\colon Y\arrow[rr]\arrow[dr] & & X\arrow[dl] \\
            & \Delta &
        \end{tikzcd}
    \]
    be a non-constant family of stable maps over the unit disk $\Delta$, where $Y$ is smooth over $\Delta^* = \Delta\backslash \{0\}$ and $X$ is a smooth projective threefold, which is a flat family of surfaces over $\Delta$.
\end{Hypothesis}

\begin{Lemma}\label{K3ISOTRIVIALLEMNODALSING}
    Let $Y/\Delta$ be a generically smooth family of stable curves over the unit disk $\Delta$.
    \begin{enumerate}
        \item $Y$ has canonical singularities and $K_Y$ is Cartier.
        \item For any node $p \in Y_0$ formed by the intersection of two local branches $C_1, C_2$ of $Y_0$, there exists $m \ge 1$ such that locally at $p$, $Y \cong \{xy = t^m\} \subset \Delta_{xyt}^3$, i.e., $Y$ has an $A_{m-1}$ singularity at $p$, and as Weil divisors on $Y$, we have the local intersection numbers at $p$
              \begin{equation}\label{K3ISOTRIVIALE400_LOCAL}
                  C_1 \cdot C_2 = -C_1^2 = -C_2^2 = \frac{1}{m}.
              \end{equation}
        \item The family admits a minimal resolution of singularities $\psi \colon \widehat{Y} \to Y$ which is an isomorphism outside the singular locus of $Y_0$. For each $A_{m-1}$ singularity $p$, the exceptional locus $\psi^{-1}(p)$ is a chain of $m-1$ smooth rational $(-2)$-curves, and $K_{\widehat{Y}} = \psi^* K_Y$.
    \end{enumerate}
\end{Lemma}
\begin{proof}
    Since $Y \to \Delta$ is a flat family of stable curves and the general fibre $Y_t$ is smooth, the only singularities of the total space occur at the nodes of $Y_0$. That the singularity is $A_{m-1}$ and that locally $Y$ is given by $xy=t^m$ around $C_1\cap C_2$ is covered, e.g., in \cite[\S 3.B, 3.C]{harrismorrison}. Because $A_{m-1}$ singularities are canonical and Gorenstein, $Y$ is normal and $K_Y$ is Cartier, and the minimal resolution $\psi$ is crepant, i.e., $K_{\widehat{Y}} = \psi^* K_Y$. The local equation implies that $m C_1$ and $m C_2$ are the vanishing of $\{x=0\}$ resp.\ $\{y=0\}$ on $Y$, and hence are Cartier divisors. Their transverse intersection at the origin yields $m C_1 \cdot C_2 = 1$. As $C_1+C_2$ is locally the fibre $Y_0$ over $\Delta$, we have $C_1 \cdot (C_1+C_2) = C_2 \cdot (C_1+C_2) = 0$, and the equalities of intersection numbers in \eqref{K3ISOTRIVIALE400_LOCAL} follow.
\end{proof}

\begin{Lemma}\label{K3ISOTRIVIALLEMIMAGE}
    Assume Hypothesis \ref{K3ISOTRIVIALTHMHYP}. Then the image $W=f(Y)$ is a Gorenstein surface, flat over $\Delta$.

   Suppose that $f$ is moreover a closed embedding over $\Delta^*$. Then $W$ is normal if and only if $W_0$ is reduced.
\end{Lemma}
\begin{proof}
    The image $W=f(Y)$ is a Cartier divisor in the smooth complex threefold $X$ so is Gorenstein and hence Cohen-Macaulay. As it is a family of proper curves, it is flat over $\Delta$. The morphism $f$ is proper birational onto its image and the central fibres, viewed as Cartier divisors $Y_0$ and $W_0$ pulled back from $t=0$ on $\Delta$, must satisfy the pushforward relation of cycles
    \[f_*[Y_0] = [W_0].\]
    From stability, $Y_0$ is a reduced curve, so $[Y_0] = \sum [C_i]$ where each component $C_i$ occurs with multiplicity $1$. For any component $C_i$ mapping dominantly to a curve $\Gamma \subset W_0$, it does so with degree $d_i \ge 1$, contributing $d_i [\Gamma]$ to $[W_0]$. Any component contracted by $f_0$ maps to a point and contributes $0$ to $[W_0]$.

    If $W$ is normal, then for every irreducible component $\Gamma \subset W_0$, the local ring $\OO_{W,\eta_\Gamma}$ is a DVR. Since $Y$ and $W$ are birational, we have $k(W)\cong k(Y)$, so if $C_i\subset Y_0$ is an irreducible component mapping dominantly to $\Gamma$, we have $\OO_{W,\eta_\Gamma} \subseteq \OO_{Y,\eta_{C_i}}$ and so they are isomorphic. In particular, the uniformiser has the same valuation $1$ on each ring and this is necessarily also the coefficient of $\Gamma$ in the divisor $W_0\in\Pic(W)$ and the degree of the morphism $f_0|_{C_i}\colon C_i\to \Gamma$. Moreover, exactly one component $C_i$ maps dominantly onto $\Gamma$, and it maps with degree $d_i=1$. Any other component of $Y_0$ must either map to a different curve in $W_0$ or be contracted to a point. In particular, $W_0$ is reduced.

    Conversely if every fibre of $W/\Delta$ is reduced, then $W$ is normal by Serre's criterion: Cohen--Macaulay, hence $S_2$, and $R_1$ since $\operatorname{Sing}(W)$ is contained in the finite singular locus of the reduced curve $W_0$, a set of codimension $2$ in the surface $W$.
\end{proof}

\subsection{A birational local flatness result}

Given a fibred surface $W$ in a threefold $X$, and a curve in the surface dominating the base, we would like to know whether blowing up the curve resolves in any way the singularities of the surface along the curve. Even though this statement is not exactly true as stated (see the example below), it is true on a local branch of the surface as we repeatedly blow up.

\begin{Lemma}\label{K3ISOTRIVIALE601}
Let $W$ be a hypersurface in $X = \Delta_{xyt}^3$ such that $W$ is smooth over $\Delta_t^*$ and let $\Gamma\subset W$ be a section of $W/\Delta_t$.
    Let
    \[
        \begin{tikzcd}
            X = X_{0} & X_{1} \ar{l} & ...\ar{l} & X_{n} \ar{l} & ... \ar{l}
        \end{tikzcd}
    \]
    be a sequence of blowups such that $X_{n+1}\to X_{n}$ is the blowup of $X_{n}$ along $\Gamma_{n}$, where $W_n$ is the proper transform of $W$ under $X_{n}\to X$ and $\Gamma_n \subset W_n$ is the proper transform of $\Gamma$ under $W_n\to W$. Then there exists $n_0$ such that for all $n\ge n_0$, $W_n$ has a smooth locally irreducible component in an analytic open neighbourhood of $\Gamma_n$.
\end{Lemma}
\begin{proof}
    We may assume that $\Gamma = \{x = y = 0\}$ and $W = \{f(x,y,t) = 0\}$ for some $f(x,y,t) \in \CC[[x,y,t]]$.
    Since $\Gamma\subset W$, $f(0,0,t)\equiv 0$. Since $W_t$ is smooth for $t\ne 0$, either $f_x(0,0,t)\not\equiv 0$ or $f_y(0,0,t)\not\equiv 0$. So after a change of coordinates, we may write the defining equation of $W$ as
    \begin{equation}\label{K3ISOTRIVIALE807}
        f(x,y,t) = t^a y + O(x^2,xy,y^2) = 0
    \end{equation}
    for some $a\ge 0$, where we use the notation $O(f_1,f_2,...,f_k)$ to denote an element in the ideal generated by $f_1,f_2,...,f_k$ in $\CC[[x,y,t]]$.

    If $a = 0$, then $W$ is smooth and we are done. We will show that $a$ decreases during blowups and eventually we have $a=0$.

    Write $f=(t^a+g)y+B(x,t)$ where $g(x,y,t) \in \CC[[x,y,t]]$ satisfies $g(0,0,t)\equiv0$, and $B(x,t)=f(x,0,t)\in(x^2)$.
    We can perform a coordinate change $y \mapsto y - \delta(x,t)$ with $\delta \in (x^2)$ to absorb the  $x$-dependent terms of $B(x,t)$ into the $y$ factor.
    Since $B(x,t) \in (x^2)$ and $\delta(x,t) \in (x^2)$, any remainder must also be in $(x^2)$. Therefore, if the remainder is non-zero, its lowest power of $x$ is $x^m$ for some $m \ge 2$.
    Thus, in the new coordinates, we have either
    \begin{equation}\label{K3ISOTRIVIALE820}
        f(x,y,t) = (t^a + g(x,y,t))y
    \end{equation}
    or 
    \begin{equation}\label{K3ISOTRIVIALE821}
        f(x,y,t) = (t^a + g(x,y,t))y + x^m h(x,t)
    \end{equation}
    for some $m\ge 2$, $g(x,y,t)\in \CC[[x,y,t]]$ and $h(x,t)\in \CC[[x,t]]$ satisfying $g(0,0,t)\equiv 0$, $h(0,t)\not\equiv 0$ and
    \begin{equation}\label{K3ISOTRIVIALE805}
        \dfrac{\partial^a h}{\partial t^a} \equiv 0.
    \end{equation}
    For example, if $f = t^2 y + x^2 t^3 + x^2 t$, we have $a=2$ and $B(x,t) = x^2 t^3 + x^2 t$. By defining $\delta(x,t) = x^2 t$ and performing the coordinate shift $y \mapsto y - \delta(x,t)$, the equation becomes $t^2(y - x^2 t) + x^2 t^3 + x^2 t = t^2 y + x^2 t$. The remainder $x^2 t$ has $t$-degree strictly less than $2$, i.e., \eqref{K3ISOTRIVIALE805} holds for $h(x,t)=t$.
    
    In the case of \eqref{K3ISOTRIVIALE820}, the component $V = \{y = 0\}\subset W$ contains $\Gamma$ and is smooth along $\Gamma$. The same holds for the proper transforms of $V$.
    
    In the case of \eqref{K3ISOTRIVIALE821}, after we blow up $X$ along $\Gamma$, we obtain $W_1$ given by
    \[
        (t^a + g(x, wx, t)) w + x^{m-1} h(x,t) = 0
    \]
    in $\Delta_{xwt}^3$, locally around $\Gamma_1 = \{x = w = 0\}$. After a further $m-2$ blowups, we obtain $W_{m-1}$ given by
    \begin{equation}\label{K3ISOTRIVIALE806}
        (t^a + g(x, w^{m-1} x, t)) w + x h(x,t) = 0.
    \end{equation}
    We can rewrite the left hand side of \eqref{K3ISOTRIVIALE806} as
    \[
        (t^a + g(x, w^{m-1} x, t)) w + x h(x,t)
        = t^a w + x h(0,t) + O(x^2, xw, w^2).
    \]
    Let $b = \operatorname{mult}_0 h(0,t)$ be the order of $h(0,t)\in \CC[[t]]$ in $t$.
    By \eqref{K3ISOTRIVIALE805}, $b < a$. We can write $h(0,t) = t^b u(t)$ where $u(0)\ne 0$. The linear terms of our equation are $t^a w + x t^b u(t) = t^b (t^{a-b}w + x u(t))$. By introducing new coordinates $w' = t^{a-b}w + x u(t)$ and $x' = w$, we can rewrite the defining equation of $W_{m-1}$ as
    \begin{equation}\label{K3ISOTRIVIALE808}
        t^b w' + O((x')^2, x' w', (w')^2) = 0.
    \end{equation}
    (In particular, if $b=0$, this equation is simply $w' + O((x')^2, x' w', (w')^2) = 0$, meaning $W_{m-1}$ is smooth). Comparing \eqref{K3ISOTRIVIALE807} and \eqref{K3ISOTRIVIALE808}, we see that the exponent $a$ strictly decreases. Repeating this argument proves the lemma.
\end{proof}

\begin{Example}\label{K3ISOTRIVIALEXAMPLEFABLE}
    An earlier version of the lemma above asserted the stronger conclusion that the strict transform $W_m$ becomes \emph{smooth along $\Gamma_m$} after finitely many blowups. This is false, and the need of working with an analytic subvariety to extract a smooth component is illustrated by the following counterexample. We thank Anthropic's Fable LLM for finding this counterexample.
    
    Let
    \[
        W_0 = \{\, t^2 x + x^2 + t\, y^3 = 0 \,\} \subset \Delta_{xyt}^3, \qquad \Gamma_0 = \{x = y = 0\}.
    \]
    This satisfies all the hypotheses: $W_0$ is flat over $\Delta_t$ with smooth general fibres, and $\Gamma_0$ is a section. 
    However, the strict transform $W_n$ remains singular at the section point $\Gamma_n\cap (W_n)_0$ for \emph{every} $n\ge 0$. 
    
    To see this, let us perform the first two blowups. 
    First, blowing up $\Gamma_0 = \{x=y=0\}$, in the chart $x = uv, y = v$ with exceptional divisor $v=0$, the strict transform $W_1$ has equation
    \[
        t^2 u + u^2 v + t v^2 = 0.
    \]
    Intersecting with $v=0$ gives the new section $\Gamma_1 = \{u=v=0\}$. 
    Blowing up $\Gamma_1$, in the chart $u = ab, v = b$ with exceptional divisor $b=0$, the strict transform $W_2$ has equation
    \[
        t^2 a + a^2 b^2 + t b = 0,
    \]
    with the new section $\Gamma_2 = \{a=b=0\}$. 
    At the origin $(a,b,t)=(0,0,0)$ of this chart, the lowest degree term is $tb$ (degree 2). Thus $W_2$ has multiplicity 2 and its tangent cone is cut out by $tb = 0$. 
    Every subsequent blowup will reproduce a singularity of the exact same form so $W_n$ never becomes smooth along the strict transform. Geometrically, the central fibre $(W_n)_0$ breaks into multiple non-reduced components, and $\Gamma_n$ always passes through their intersection.

    However, in the lemma above, in the analytic local ring, $W_2$ splits into components. Solving the quadratic equation $a^2 b^2 + t b + t^2 a = 0$ for $b$ as a formal power series yields $b = -ta - ta^4 - 2ta^7 - \ldots$. Thus, the local equation factors as
    \[
        (b + ta + ta^4 + \ldots)(a^2 b + t - a^3 t - \ldots) = 0.
    \]
    The first factor defines a component containing $\Gamma_2$ which is smooth.
\end{Example}

We prove now a relative version of the above. For a family of smooth threefolds $X/U$, a family of fibred surfaces $W/U$ in $X$, and a family of sections $\Gamma/U$ as in the setup above, blowing up $\Gamma_u$ may introduce singularities to $X_u$, depending on the singularities of the central fibres $W_{0,u}$ of the fibration $W_u\to\Delta_t$. At any fixed blowup of local irreducible components of strict transforms of $\Gamma$, we can remove a closed subset from $U$ to ensure $X_u$ is smooth outside this locus, but repeating the process it is not clear that one open subset of $U$ exists which will work for all blowups. The following proves precisely this is possible. 

\begin{Proposition}\label{K3ISOTRIVIALLEMCARTIER}
    Let $W$ be a hypersurface in $X = \Delta_{xyt}^3\times U$ for $U = \Delta^\ell$ such that $W$ is smooth over $\Delta_t^* \times U$ and let $\Gamma\subset W$ be a section of $W/(\Delta_t\times U)$.
    Let
    \[
        \begin{tikzcd}
            X = X_{0} & X_{1} \ar{l} & ...\ar{l} & X_{n} \ar{l} & ... \ar{l}
        \end{tikzcd}
    \]
    be a sequence of blowups where
    \begin{itemize}
    \item $X_1\to X_0$ is the blowup of $X_0 = X$ along $\Gamma$ with exceptional divisor $E_1$ and proper transform $W_1\subset X_1$ of $W_0 = W$;
    \item $X_2\to X_1$ is the blowup of $X_1$ along $\Gamma_1$, which is the irreducible component of $W_1\cap E_1$ dominating $\Delta_t\times U$, with exceptional divisor $E_2$ and proper transform $W_2\subset X_2$ of $W_1$;
    \item $X_n\to X_{n-1}$ is the blowup of $X_{n-1}$ along $\Gamma_{n-1}$, which is the irreducible component of $W_{n-1}\cap E_{n-1}$ dominating $\Delta_t\times U$, with exceptional divisor $E_n$ and proper transform $W_n\subset X_n$ of $W_{n-1}$.
    \end{itemize} 
    Then there exists a nonempty Zariski open set $U^\circ\subset U$ such that $X_m$ is smooth over $U^\circ$, $W_m$ is flat over $U^\circ$, and $\Gamma_m$ is a section of $X_m$ over $\Delta_t\times U^\circ$ for all $m$.
\end{Proposition}
\begin{proof}
By removing a closed subset at every blowup, it is clear that such $U^\circ$ exists as the complement of countably many proper closed subvarieties of $U$. The point is to prove that it is the complement of finitely many proper closed subvarieties of $U$.

Let $X_u$, $W_u$ and $\Gamma_u$ be the fibres of $X/U$, $W/U$ and $\Gamma/U$ over a point $u\in U$, respectively. For a general point $u\in U$, from Lemma \ref{K3ISOTRIVIALE601}, there exists $n_0\in \ZZ^+$ such that $W_{n_0,u}$ has a smooth locally irreducible component along $\Gamma_{n_0,u}$, where $W_{n_0,u}\subset W_{n_0}$ and $\Gamma_{n_0, u}\subset \Gamma_{n_0}$ are the proper transforms of $W_u$ and $\Gamma_u$ under $X_n\to X$ and $W_n\to W$, respectively.

There exists a nonempty Zariski open set $U^\circ$ of $U$ such that $X_m$ is smooth over $U^\circ$, $W_m$ is flat over $U^\circ$, and $\Gamma_m$ is a section of $X_m$ over $\Delta_t\times U^\circ$ for all $m\le n_0$. For simplicity, we replace $U$ with $U^\circ$ and assume that $X_m$ is smooth over $U$, $W_m$ is flat over $U$, and $\Gamma_m$ is a section of $X_m$ over $\Delta_t\times U$ for all $m\le n_0$.

Since $W_{n_0,u}$ has a smooth locally irreducible component along $\Gamma_{n_0,u}$ for a general point $u\in U$, $W_{n_0}$ has a smooth locally irreducible component along $\Gamma_{n_0}$. That a single $n_0$ suffices over a dense Zariski open, rather than fibrewise with $n_0 = n_0(u)$ can be seen by running the argument in the proof of Lemma \ref{K3ISOTRIVIALE601} over the generic point of $U$ and spreading the resulting smooth component out. All conclusions below are local on $\Delta_t\times U$. In an analytic open neighbourhood of $\Gamma_{n_0}$, $X_{n_0}\cong \Delta_{xyt}^3\times U$ and we assume that
$\Gamma_{n_0} = \{x = y = 0\}$. Then $W_{n_0}$ is given by
\[
g(x,y,u,t) y = 0
\]
for some $g(x,y,u,t) \in \BC[[x,y,u,t]]$ with $V_{n_0} = \{y = 0\}$ the smooth component of $W_{n_0}$ containing $\Gamma_{n_0}$. Since $W_{n_0}$ is smooth over $\Delta_t^* \times U$, $V_{n_0}$ is the only component containing $\Gamma_{n_0}$ and so $g(0,0,u,t)\not\equiv 0$. Then $W_{n_0+k}$ is given by
\[
g(x, x^k w,u, t) w = 0
\]
in an open neighbourhood of $\Gamma_{n_0 + k} = \{x = w = 0\}$ for all $k\ge 0$. This implies that $X_m$ is smooth over $U$, $W_m$ is flat over $U$, and $\Gamma_m$ is a section of $X_m$ over $\Delta_t\times U$ for all $m\ge n_0$ and hence for all $m$.
\end{proof}

\section{The main theorem}

The following result will imply Theorems \hyperref[thm:A]{A} and \hyperref[thm:B]{B}. We state it in as much generality as the proof affords, though the main case of interest is $X=S\times D$ the trivial family for $S$ a smooth projective surface with nef canonical divisor.

\begin{Theorem}\label{K3ISOTRIVIALTHMNORMAL}
    Let $C$ and $D$ be smooth projective curves, let $X$ be a flat projective family of surfaces over $D$, and
    let $f\colon Y=C\times D\dashrightarrow X$ be a rational map over $D$ such that
    \begin{itemize}
        \item $g(C)\ge 2$;
        \item $X$ is smooth;
        \item $f$ maps $Y$ birationally onto its image $W = f(Y)$;
        \item $W$ is normal;
        \item $K_X+W$ is relatively nef on $W$ over $D$;
        \item there exists $n_0$ such that
              \begin{equation}\label{K3ISOTRIVIALE801}
                  \HH^i(X_p, n(K_X + W) - W) = \HH^i(X_p, n(K_X + W)) = 0
              \end{equation}
              for all $i \ge 1$, $n\ge n_0$ and $p\in D$, where $X_p$ is the fibre of $X/D$ over $p\in D$;
        \item there exists an open set $Y^\circ\subset Y$ such that
              $Y\backslash Y^\circ$ consists of finitely many points, $f$ is regular on $Y^\circ$ and for every $p\in D$ and every $0$-dimensional subscheme $V$ of $X_p$ supported on $f(Y^\circ)\cap X_p$, there exists $n_1$, depending on the length of $V$, with the property that
              the map
              \begin{equation}\label{K3ISOTRIVIALE802}
                  \begin{tikzcd}
                      \HH^0(\OO_{X_p}(n(K_X + W))) \ar[two heads]{r} &
                      \HH^0(\OO_{V}(n(K_X + W)))
                  \end{tikzcd}
              \end{equation}
              is surjective for all $n\ge n_1$.
    \end{itemize}
    Then there does not exist $p\in D$ such that $W_p$ is integral and singular.
\end{Theorem}

We begin with some remarks clarifying the various assumptions.

\begin{Remark}
    \begin{enumerate}
        \item Note that both \eqref{K3ISOTRIVIALE801} and \eqref{K3ISOTRIVIALE802} are implied by $K_X+W$ being relatively ample. Also, both of them combined imply that
        \begin{equation}\label{K3ISOTRIVIALE803}
            \HH^i(\OO_{X_p}(n(K_X + W)) \otimes I_V) = 0
        \end{equation}
        for all $i\ge 1$, where $I_V$ is the ideal sheaf of $V$ in $X_p$.
        \item From Lemma \ref{K3ISOTRIVIALLEMIMAGE}, $W$ is normal if and only if $W_p$ is reduced for all $p\in D$.
    \end{enumerate}
\end{Remark}

\begin{proof}[Proof of {\hyperref[thm:A]{Theorem~A}}]
    Since $H$ is base point free, its generic member is smooth by Bertini's theorem. Since it is primitive and the Picard rank is one, all curves in the linear system are irreducible. The assumption that $K_S+H$ is nef guarantees that the generic curve has genus $g\ge 2$: writing $K_S=aH$ (as $\NS(S)=\ZZ H$), adjunction gives $2g-2=(K_S+H)\cdot H=(a+1)H^2$, so nefness forces $a\ge-1$. The case $a=-1$ is impossible, as it would make $-K_S=H$ an ample generator, forcing $S$ to be a Picard-rank-one del Pezzo, i.e.\ $\PP^2$, on which $-K_{\PP^2}=\OO(3)$ is not a generator. Thus $a\ge0$, $g\ge2$, and $K_S+H=(a+1)H$ is ample. 
    
    Note next that $\dim |H| \geq 2$, and if equality holds, the ampleness of $H$ implies the induced map $S \to |H|^\vee \cong \PP^2$ is a finite morphism. Since $S \not\cong \PP^2$, its degree $d$ is strictly greater than 1. As $S$ is smooth and $\PP^2$ is simply connected, the map must be ramified along a nonempty branch curve $B \subset \PP^2$. A curve in $|H|$ corresponds to the pullback of a line in $\PP^2$ and any line tangent to the branch locus $B$ pulls back to a singular curve in $S$, showing the discriminant locus $\Delta$ is nonempty. If $\dim |H| \ge 3$, the nonemptiness of $\Delta$ follows similarly by considering Lefschetz pencils.
    
    We prove now that $\mu$ is quasi-finite, i.e.\ that every fibre of $\mu$ is finite. Suppose for a contradiction that some fibre of $\mu$ is positive-dimensional. Taking the closure of a one-dimensional component of that fibre in the projective space $|H|$, we obtain a complete curve $D \subset |H|$ such that the intersection $U = D \cap |H|_{\sm}$ is a dense open subset of $D$ on which $\mu$ is constant. The family of curves, embedded as a surface over $D$ in $X = S \times D$, is birational to $C\times D$, possibly after a finite base change $\widetilde D\to D$ trivialising this family. Since $D$ is a complete curve in the projective space $|H|$, it must intersect the nonempty discriminant hypersurface $\Delta$, and so the family contains singular fibres.

    We now apply Theorem \ref{K3ISOTRIVIALTHMNORMAL} to the family over $D$. Restricted to fibres, $K_X+W$ is $K_S+H$, which is ample and hence relatively nef. As $K_S+H$ is ample, asymptotic vanishing and very ampleness for large $n$ hold, i.e., conditions \eqref{K3ISOTRIVIALE801} and \eqref{K3ISOTRIVIALE802}. The theorem then guarantees that no fibre can be an integral singular curve. This contradicts the necessity of singular fibres on $D$, forcing every fibre of $\mu$ to be finite; that is, $\mu$ is quasi-finite.
\end{proof}

\begin{Remark}\label{rem:properness}
    \hyperref[thm:A]{Theorem~A} asserts only that $\mu$ is \emph{quasi-finite}, and this cannot in general be strengthened to the assertion that $\mu$ is \emph{finite}, i.e.\ proper. Properness asks whether $\mu$ extends over the boundary: by the valuative criterion applied to the proper space $|H|$, it fails exactly when a one-parameter family $\{C_t\}_{t\ne0}$ of smooth members has a singular flat limit $C_0\in|H|$ while the moduli $[C_t]$ nonetheless converge to an interior point of $\mathcal M_g$, i.e.\ the stable limit of the $C_t$ is a smooth curve of genus $g$.
    Whether the stable limit is smooth is governed by the singularities of $C_0$ (see \cite[Ch.~3]{harrismorrison}). We expect that, even under the assumption $\NS(S)=\ZZ H$, there exist irreducible singular curves in $|H|$ with smooth stable reduction.
\end{Remark}

The rest of this section is concerned with a proof of Theorem \ref{K3ISOTRIVIALTHMNORMAL}. 

We first resolve the indeterminacy of the rational map $f$ with the commutative diagram
\begin{equation}\label{K3ISOTRIVIALE800}
    \begin{tikzcd}
        Z \ar{r}{\phi}\ar{d}[left]{\psi} & X\\
        Y \ar[dashed]{ru}[below right]{f}
    \end{tikzcd}
\end{equation}
where $\psi\colon Z\to Y$ is a projective birational morphism. We choose $Z$ to be smooth and $\psi$ to be minimal in the sense that for every $(-1)$-curve $E$ in the exceptional locus $E_\psi$ of $\psi$, $\phi_* E \ne 0$. Note that every fibre $Z_p$ of $Z/D$ has simple normal crossing support.

Let $\pi_C\colon Z\to C$ be the projection from $Z$ to $C$ via $Z\to Y \to C$ and let $\Gamma_s = \pi_C^{-1}(s)$ for $s\in C$. Clearly, outside of finitely many points $C\backslash C^\circ$, $\Gamma_s$ is a section of $Z/D$ for $s\in C^\circ$. For every $s\in C^\circ$, $\phi(\Gamma_s)$ is a section of $\pi\colon X\to D$. We choose $g+1 = g(C)+1$ distinct points $s_1,s_2,...,s_g, s_{g+1}$ on $C^\circ$ and $m\in \ZZ^+$ and let
\[
    \sigma = s_1 + s_2 + ... + s_g + m s_{g+1} \in \Pic(C).
\]

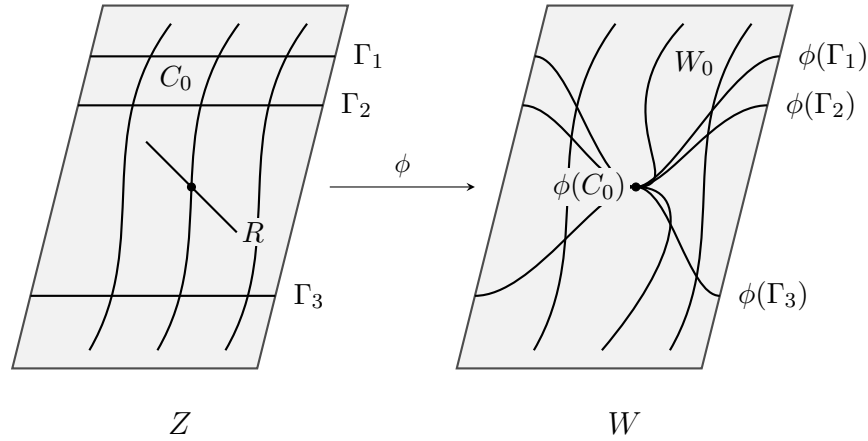
\begin{figure}[htbp]
    \centering
    \begin{tikzpicture}[>=stealth, x=1.2cm, y=1.2cm]

        \begin{scope}[shift={(0,0)}]
            \node[font=\large] at (0.5, -2.6) {$Z$};

            \filldraw[fill=black!5, draw=black!70, thick]
            (-1.35, -2) -- (1.35, -2) -- (2.35, 2) -- (-0.35, 2) -- cycle;

            \draw[thick] (-0.5, -1.8) .. controls (0.2, -0.6) and (-0.5, 0.6) .. (0.4, 1.8);
            \draw[thick] (1.0, -1.8) .. controls (1.7, -0.6) and (1.0, 0.6) .. (1.9, 1.8);

            \draw[thick] (-0.49, 1.44) -- (2.21, 1.44) node[right, xshift=5pt, fill=white, inner sep=1.5pt] {$\Gamma_1$};
            \draw[thick] (-0.625, 0.9) -- (2.075, 0.9) node[right, xshift=5pt, fill=white, inner sep=1.5pt] {$\Gamma_2$};
            \draw[thick] (-1.15, -1.2) -- (1.55, -1.2) node[right, xshift=5pt, fill=white, inner sep=1.5pt] {$\Gamma_3$};

            \draw[thick] (0.25, -1.8) .. controls (0.95, -0.6) and (0.25, 0.6) .. (1.15, 1.8) node[pos=0.85, left, xshift=-5pt, yshift=-2pt, fill=black!5, inner sep=1.5pt] {$C_0$};
            \draw[thick] (1.125, -0.5) -- (0.125, 0.5) node[pos=0.0, right, fill=black!5, inner sep=1.5pt] {$R$};
            \filldraw (0.625, 0.0) circle (1.5pt);

        \end{scope}

        \draw[->] (2.15, 0) -- (3.75, 0) node[midway, above] {\small $\phi$};

        \begin{scope}[shift={(4.9,0)}]
            \node[font=\large] at (0.5, -2.6) {$W$};

            \filldraw[fill=black!5, draw=black!70, thick]
            (-1.35, -2) -- (1.35, -2) -- (2.35, 2) -- (-0.35, 2) -- cycle;

            \draw[thick] (-0.5, -1.8) .. controls (0.2, -0.6) and (-0.5, 0.6) .. (0.4, 1.8);
            \draw[thick] (1.0, -1.8) .. controls (1.7, -0.6) and (1.0, 0.6) .. (1.9, 1.8);

            \draw[thick] (-0.49, 1.44) .. controls (-0.1, 1.44) and (0.2, 0.0) .. (0.625, 0.0)
                         .. controls (1.0, 0.0) and (1.7, 1.44) .. (2.21, 1.44) node[right, xshift=5pt, fill=white, inner sep=1.5pt] {$\phi(\Gamma_1)$};
            \draw[thick] (-0.625, 0.9) .. controls (-0.2, 0.9) and (0.2, 0.0) .. (0.625, 0.0)
                         .. controls (1.0, 0.0) and (1.5, 0.9) .. (2.075, 0.9) node[right, xshift=5pt, fill=white, inner sep=1.5pt] {$\phi(\Gamma_2)$};
            \draw[thick] (-1.15, -1.2) .. controls (-0.5, -1.2) and (0.2, 0.0) .. (0.625, 0.0)
                         .. controls (1.0, 0.0) and (1.2, -1.2) .. (1.55, -1.2) node[right, xshift=5pt, fill=white, inner sep=1.5pt] {$\phi(\Gamma_3)$};

            \draw[thick] (0.25, -1.8) .. controls (1.1, -0.8) and (1.25, 0.0) .. (0.625, 0.0)
                        .. controls (1.25, 0.0) and (0.2, 0.8) .. (1.15, 1.8) node[pos=0.85, right, xshift=5pt, fill=black!5, inner sep=1.5pt] {$W_0$};
            
            \filldraw (0.625, 0.0) circle (1.5pt) node[left, xshift=-2pt, fill=black!5, inner sep=1.5pt] {$\phi(C_0)$};

        \end{scope}

    \end{tikzpicture}
    \caption{Example of a birational morphism $\phi$. The central fibre of $Z$ consists of a curve $C_0 \cong C$ and a rational tail $R$. The curve $C_0$ is contracted to the singularity of $W_0$, $R$ maps dominantly onto the curve $W_0$, and the sections all pass through the singularity.}
    \label{fig:Z_to_W}
\end{figure}

We write
\[
    \Gamma_\sigma = \Gamma_{s, m} = \pi_C^* (\sigma) = \Gamma_{s_1} + \Gamma_{s_2} + ... + \Gamma_{s_g} + m \Gamma_{s_{g+1}} \in \Pic(Z)
\]
and let $I_{\phi(\Gamma_\sigma)}$ be the ideal sheaf of the Weil divisor $\phi(\Gamma_\sigma)$ in $W$. Since $W$ is normal, we have
\begin{align}
    \HH^0(W, L) &= \HH^0(Z, \phi^* L) \\
    \HH^0(W, L\otimes I_{\phi(\Gamma_\sigma)})
    &= \HH^0(Z, \phi^* L\otimes \OO_Z(-\Gamma_\sigma))\label{K3ISOTRIVIALEEQNORMAL}
\end{align}
for all line bundles $L$ on $W$.

We claim

\begin{Lemma}\label{K3ISOTRIVIALTHMNORMALCLAIM000}
    There exists a nonempty Zariski open set $U\subset C^{g+1}$ such that for every fixed $m\in \ZZ^+$,
    \[
        h^0(W, \OO_W(n(K_X + W) + \pi^*A)\otimes I_{\phi(\Gamma_{s,m})})
    \]
    is constant for all $s = (s_1,s_2,...,s_{g+1})\in U$, when
    $\deg A\gg n \gg m$, where $A\in \Pic(D)$ is a divisor on $D$.
\end{Lemma}

The essential point of the above lemma is that $U$ exists as a Zariski open set, i.e., $C^{g+1}\backslash U$ is a proper subvariety of $C^{g+1}$. The lemma trivially holds over $U$ if we further assume that $K_X+W$ is relatively ample on $W$ over $D$, where $C^{g+1}\backslash U$ is a union of countably many proper subvarieties of $C^{g+1}$. The key in proving the above will be Proposition \ref{K3ISOTRIVIALLEMCARTIER}.

\begin{Lemma}\label{K3ISOTRIVIALTHMNORMALCLAIM001}
    Let $U$ be the Zariski open set obtained in Lemma \ref{K3ISOTRIVIALTHMNORMALCLAIM000}. Then there exist
    $s = (s_1,s_2,...,s_{g+1}), t=(t_1,t_2,...,t_{g+1})\in U$
    and $m\in \ZZ^+$ such that
    \begin{equation}\label{K3ISOTRIVIALE809}
        \begin{aligned}
             & \quad h^0(Z, \phi^* (n(K_X + W) + \pi^* A) - \Gamma_{s,m}) \\
             & \ne h^0(Z, \phi^* (n(K_X + W) + \pi^* A) - \Gamma_{t,m})
        \end{aligned}
    \end{equation}
    for $n\in \ZZ$ and $A\in \Pic(D)$ with $\deg A \gg n\gg m$, if $W_p$ is integral and singular for some $p\in D$.
\end{Lemma}
This will boil down to Proposition \ref{K3ISOTRIVIALPROPJACOBIAN}.

From \eqref{K3ISOTRIVIALEEQNORMAL} the above two lemmas contradict the assumption that $W_p$ is integral and singular and imply our theorem.
\begin{proof}[Proof of Lemma \ref{K3ISOTRIVIALTHMNORMALCLAIM000}]

Fixing $g+1$ distinct points $s_1,s_2,...,s_{g+1}$, we construct a sequence of blowups
\[
    \begin{tikzcd}
        X = X_{s,0} & X_{s,1} \ar{l} & ...\ar{l} & X_{s,g} \ar{l} & X_{s,g+1} \ar{l} & ... \ar{l}
    \end{tikzcd}
\]
where $X_{s,n}\to X_{s,n-1}$ is the blowup of $X_{s,n-1}$ along the proper transform of $\phi(\Gamma_{s_n})$ under $X_{s,n-1}\to X$ for $n=0,1,...,g+1$ and it is the blowup of $X_{s,n-1}$ along the proper transform of $\phi(\Gamma_{s_{g+1}})$ under $W_{s,n-1}\to W$ for $n\ge g+2$, with $W_{s,k}$ the proper transform of $W$ under $X_{s,k}\to X$. Note that all proper transforms $W_{s,n}\subset X_{s,n}$ of $W$ have connected fibres over $W$.

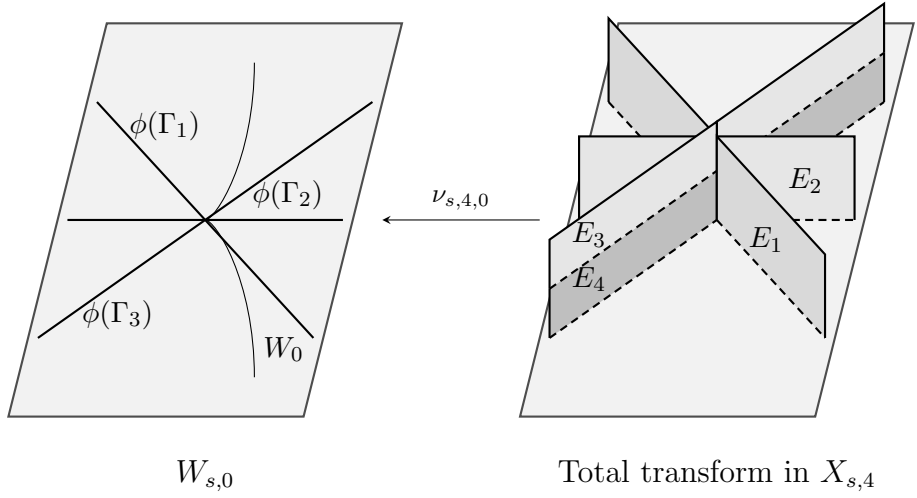
\begin{figure}[htbp]
    \centering
    \begin{tikzpicture}[>=stealth, x=1.3cm, y=1.3cm]

        \begin{scope}[shift={(0,0)}]
            \node[font=\large] at (0.5, -2.6) {$W_{s,0}$};

            \filldraw[fill=black!5, draw=black!70, thick]
            (-1.5, -2) -- (1.5, -2) -- (2.5, 2) -- (-0.5, 2) -- cycle;

            \draw[thin] (1.0, 1.6) .. controls (1.0, 0.6) and (0.6, 0.0) .. (0.5, 0.0)
                        .. controls (0.6, 0.0) and (1.0, -0.6) .. (1.0, -1.6) node[pos=0.9, right] {$W_0$};

            \draw[thick] (-0.6, 1.2) -- (1.6, -1.2) node[pos=0.1, right] {$\phi(\Gamma_1)$};
            \draw[thick] (-0.9, 0.0) -- (1.9, 0.0) node[pos=0.8, above] {$\phi(\Gamma_2)$};
            \draw[thick] (-1.2, -1.2) -- (2.2, 1.2) node[pos=0.1, right] {$\phi(\Gamma_3)$};

        \end{scope}

        \draw[<-] (2.3, 0) -- (3.9, 0) node[midway, above] {\small $\nu_{s,4,0}$};

        \begin{scope}[shift={(5.2,0)}]
            \node[font=\large] at (0.5, -2.6) {Total transform in $X_{s,4}$};

            \filldraw[fill=black!5, draw=black!70, thick]
            (-1.5, -2) -- (1.5, -2) -- (2.5, 2) -- (-0.5, 2) -- cycle;

            \fill[black!15]
            (-0.6, 1.2) -- (0.5, 0) -- (0.5, 0.85) -- (-0.6, 2.05) -- cycle;
            \draw[thick] (0.5, 0) -- (0.5, 0.85) -- (-0.6, 2.05) -- (-0.6, 1.2);
            \draw[thick, densely dashed] (-0.6, 1.2) -- (0.5, 0);

            \fill[black!25]
            (0.5, 0) -- (2.2, 1.2) -- (2.2, 1.7) -- (0.5, 0.5) -- cycle;
            \draw[thick] (2.2, 1.2) -- (2.2, 1.7);
            \draw[thick] (0.5, 0.5) -- (0.5, 0);
            \draw[thick, densely dashed] (0.5, 0) -- (2.2, 1.2);

            \fill[black!10]
            (0.5, 0.5) -- (2.2, 1.7) -- (2.2, 2.2) -- (0.5, 1.0) -- cycle;
            \draw[thick] (2.2, 1.7) -- (2.2, 2.2) -- (0.5, 1.0) -- (0.5, 0.5);
            \draw[thick, densely dashed] (0.5, 0.5) -- (2.2, 1.7);

            \fill[black!10]
            (-0.9, 0.0) -- (1.9, 0.0) -- (1.9, 0.85) -- (-0.9, 0.85) -- cycle;
            \draw[thick] (1.9, 0.0) -- (1.9, 0.85) -- (-0.9, 0.85) -- (-0.9, 0.0);
            \draw[thick, densely dashed] (-0.9, 0.0) -- (1.9, 0.0);
            \node at (1.4, 0.4) {$E_2$};

            \fill[black!25]
            (-1.2, -1.2) -- (0.5, 0) -- (0.5, 0.5) -- (-1.2, -0.7) -- cycle;
            \draw[thick] (0.5, 0) -- (0.5, 0.5);
            \draw[thick] (-1.2, -0.7) -- (-1.2, -1.2);
            \draw[thick, densely dashed] (-1.2, -1.2) -- (0.5, 0);
            \node at (-0.8, -0.6) {$E_4$};

            \fill[black!10]
            (-1.2, -0.7) -- (0.5, 0.5) -- (0.5, 1.0) -- (-1.2, -0.2) -- cycle;
            \draw[thick] (0.5, 0.5) -- (0.5, 1.0) -- (-1.2, -0.2) -- (-1.2, -0.7);
            \draw[thick, densely dashed] (-1.2, -0.7) -- (0.5, 0.5);
            \node at (-0.8, -0.15) {$E_3$};

            \fill[black!15]
            (0.5, 0) -- (1.6, -1.2) -- (1.6, -0.35) -- (0.5, 0.85) -- cycle;
            \draw[thick] (1.6, -1.2) -- (1.6, -0.35) -- (0.5, 0.85) -- (0.5, 0);
            \draw[thick, densely dashed] (0.5, 0) -- (1.6, -1.2);
            \node at (1.0, -0.2) {$E_1$};

        \end{scope}

    \end{tikzpicture}
    \caption{The case $g=2$ and $m=2$.}
    \label{fig:blowup_surfaces}
\end{figure}

Also observe that
\[
    K_{X_{s,n}} + W_{s,n} = \nu_{s,n,0}^*(K_X + W)
\]
where $\nu_{s,i,j}$ is the map $X_{s,i}\to X_{s,j}$.
Then
\[
    \begin{aligned}
         & \quad \HH^0(W, \OO_W(n(K_X + W) + \pi^*A)\otimes I_{\phi(\Gamma_{s,m})})                              \\
         & = \HH^0(W_{s,g+m}, \nu_{s,g+m,0}^* (n(K_X + W) + \pi^* A) - \sum_{i=1}^{g+m} \nu_{s,g+m,i}^* E_{s,i}) \\
         & = \HH^0(W_{s,g+m}, \nu_{s,g+m,0}^* (n(K_X + W) - W + \pi^* A) + W_{s,g+m})
    \end{aligned}
\]
for all $n\in \ZZ$ and $A\in \Pic(D)$,
where $E_{s,i}\subset X_{s,i}$ is the exceptional divisor of $X_{s,i}\to X_{s,i-1}$.

We have the exact sequence
\[
    \begin{tikzcd}
        0 \ar{r} & \OO_{X_{s,g+m}}(\nu_{s,g+m,0}^* (n(K_X + W) - W + \pi^* A))\\
        \ar{r} & \OO_{X_{s,g+m}}(\nu_{s,g+m,0}^* (n(K_X + W) - W + \pi^* A) + W_{s,g+m})\\
        \ar{r} & \OO_{W_{s,g+m}}(\nu_{s,g+m,0}^* (n(K_X + W) - W + \pi^* A) + W_{s,g+m}) \ar{r} & 0
    \end{tikzcd}
\]
From \eqref{K3ISOTRIVIALE801}, we have
\[
    \begin{aligned}
         & \quad \HH^i(\OO_{X_{s,g+m}}(\nu_{s,g+m,0}^* (n(K_X + W) - W + \pi^* A))) \\
         & = \HH^i(\OO_{X}(n(K_X + W) - W + \pi^* A)) = 0
    \end{aligned}
\]
for all $i\ge 1$, $n\in \ZZ$ and $A\in \Pic(D)$ with $\deg A \gg n \gg m$. From \eqref{K3ISOTRIVIALE803}, we have
\begin{equation}\label{K3ISOTRIVIALE_VANISHING}
    \begin{aligned}
         & \quad \HH^i(\OO_{X_{s,g+m}}(\nu_{s,g+m,0}^* (n(K_X + W) -  W + \pi^* A) + W_{s,g+m}))   \\
         & = \HH^i(\OO_{X_{s,g+m}}(\nu_{s,g+m,0}^* (n(K_X + W) + \pi^* A) - E))                    \\
         & = \HH^i(\OO_{X}(n(K_X + W) + \pi^* A)\otimes (\nu_{s,g+m,0})_* \OO_{X_{s,g+m}}(-E)) = 0
    \end{aligned}
\end{equation}
for all $i\ge 1$, $n\in \ZZ$ and $A\in \Pic(D)$ with $\deg A \gg n \gg m$, where
\[
    E = \sum_{i=1}^{g+m} \nu_{s,g+m,i}^* E_{s,i}.
\]
Here the penultimate equality follows from the projection formula and the degeneracy of the Leray spectral sequence, as
\[
    R^i(\nu_{s,g+m,0})_* \OO_{X_{s,g+m}}(-E) = 0 \quad \text{for } i \ge 1.
\]
The final vanishing in \eqref{K3ISOTRIVIALE_VANISHING} follows by pushing forward to $D$: condition \eqref{K3ISOTRIVIALE803} guarantees the higher direct images vanish, and Serre vanishing on the curve $D$ kills the remaining first cohomology since $\deg A \gg n$.
Thus,
\[
    \HH^i(\OO_{W_{s,g+m}}(\nu_{s,g+m,0}^* (n(K_X + W) -  W + \pi^* A) + W_{s,g+m})) = 0
\]
for all $i\ge 1$, $n\in \ZZ$ and $A\in \Pic(D)$ with $\deg A \gg n \gg m$. Therefore,
\begin{equation}\label{K3ISOTRIVIALE804}
    \begin{aligned}
         & \quad h^0(W, \OO_W(n(K_X + W) + \pi^*A)\otimes I_{\phi(\Gamma_{s,m})})           \\
         & = h^0(W_{s,g+m}, \nu_{s,g+m,0}^* (n(K_X + W) - W + \pi^* A) + W_{s,g+m})
        \\
         & = \chi(\OO_{W_{s,g+m}}(\nu_{s,g+m,0}^* (n(K_X + W) -  W + \pi^* A) + W_{s,g+m})) \\
         & = \chi(\OO_{X_{s,g+m}}(\nu_{s,g+m,0}^* (n(K_X + W) - W + \pi^* A) + W_{s,g+m}))  \\
         & \quad - \chi(\OO_{X_{s,g+m}}(\nu_{s,g+m,0}^* (n(K_X + W) - W + \pi^* A)))
    \end{aligned}
\end{equation}
for all $n\gg m$ and $A\in \Pic(D)$ with $\deg A \gg n$.

For each fixed $m$, $X_{s,g+m}$ forms a flat family over a nonempty Zariski open set $U_{g+m}$ of $C^{g+1}$. More precisely, there exist a nonempty Zariski open set $U_{g+m}$ of $C^{g+1}$ and smooth families $\calX_n$ over $U_{g+m}$ for $n=0,1,...,g+m$ together with maps
\[
    \begin{tikzcd}
        \calX_i \ar{r} \ar{d} & \calX_j \ar{ld}\\
        U_{g+m}
    \end{tikzcd}
\]
for $i\ge j$ such that $\calX_0 = X\times U_{g+m}$, the fibre of $\calX_n/U_{g+m}$ over $s\in U_{g+m}$ is $X_{s,n}$ and we have the diagram
\[
    \begin{tikzcd}
        X_{s,i} \ar{d}\ar{r}{\nu_{s,i,j}} & X_{s,j} \ar{dl}\ar{dr}\\
        \{s\} \ar{dr} & \calX_i \ar[ul, <-, crossing over] \ar{r} \ar{d} & \calX_j \ar{ld}\\
        & U_{g+m}
    \end{tikzcd}
\]
for $0\le j\le i \le g+m$. By the flatness of $\calX_n$ over $U_{g+m}$ for $n=0,1,...,g+m$, we see that the right hand side of \eqref{K3ISOTRIVIALE804} is constant for $s\in U_{g+m}$. More precisely, it is a polynomial in $n$ and $a = \deg A$, independent of $s\in U_{g+m}$.

The main point of Lemma \ref{K3ISOTRIVIALTHMNORMALCLAIM000} is that we can choose a nonempty Zariski open set $U$ of $C^{g+1}$ such that the right hand side of \eqref{K3ISOTRIVIALE804} is constant on $U$ {\em for all} $m$. Clearly, we may choose
\[
    U = \bigcap_{m\in \ZZ^+} U_{g+m}
\]
but $U$ is not necessarily a Zariski open set. A priori, such $U$ is the complement of countably many proper subvarieties of $C^{g+1}$, but
from Proposition \ref{K3ISOTRIVIALLEMCARTIER}, we may choose $U$ which is a nonempty Zariski open set. 
\end{proof}

\begin{proof}[Proof of Lemma \ref{K3ISOTRIVIALTHMNORMALCLAIM001}]
    Let
    \[
        \calL = \phi^* (n(K_X + W) + \pi^* A).
    \]
    The aim of the lemma is to compute $\HH^i(Z, \calL-\Gamma_{s,m})$ in terms of the controllable invariants of the singular fibres $Z_p$. By our hypothesis, $\calL$ is big and nef on $Z$ and $\calL - K_Z - \Gamma_{s,m}$ is big when $\deg A \gg n\gg m$. Let
    \[
        \calL - K_Z - \Gamma_{s,m} = P + \Delta
    \]
    be the Zariski decomposition of $\calL - K_Z - \Gamma_{s,m}$, where $P$ is a big and nef $\QQ$-divisor, $\Delta$ is an effective $\QQ$-divisor whose components have negative definite self intersection matrix and $P\Delta = 0$.

    \begin{Lemma}
    Let $X$ be a smooth projective surface and let $A, B$ and $C$ be divisors on $X$ such that $A$ and $B$ are nef and $nA + C$ is pseudo-effective for $n\gg 0$. Let
    $$
    nA + B + C = P_n + \Delta_n
    $$
    be the Zariski decomposition of $nA + B + C$ for $n\gg 0$ with $\QQ$-divisors $P_n$ and $\Delta_n$ such that $P_n$ is nef and $\Delta_n$ is effective. Then there exists $n_0$, independent of $B$, such that $A \Delta_n = 0$ for all $n\ge n_0$.
    \end{Lemma}
    
    \begin{proof}
    For a divisor $D = P + \Delta$ on $X$ with $P$ nef and $\Delta$ $\QQ$-effective, if there is a component $F$ of $\Delta$ such that $PF > 0$, then there exists $\varepsilon > 0$ such that $P +\varepsilon F$ is nef and $\Delta - \varepsilon F$ is effective. So we can rewrite $D$ as
    $$
    D = P + \Delta = (P+\varepsilon F) + (\Delta - \varepsilon F).
    $$
    Thus, if $D = P+\Delta$ is the Zariski decomposition of $D$, then
    $$\Delta' \ge \Delta$$
    for all decompositions $D = P' + \Delta'$ where $P'$ is nef, $\Delta'$ is effective and the components of $\Delta'$ have negative definite self intersection matrix.
    Consequently, we see that $\Delta_n \ge \Delta_{n+1}$ for $n\gg 0$. So the support $G = \supp(\Delta_n)$ of $\Delta_n$ stabilizes when $n$ is sufficiently large.
    
    It suffices to prove it for $B = 0$.
    Suppose that there is a component $F$ of $G$ such that $AF > 0$. We fix $m\gg 0$ such that $\supp(\Delta_m) = G$. Let $\varepsilon$ be the multiplicity of $F$ in $\Delta_m$. Since $AF > 0$, there exists $\ell\in \ZZ^+$ such that $\ell A + \varepsilon F$ is nef. Thus,
    $$
    \begin{aligned}
    P_{m+\ell} + \Delta_{m+\ell} &= (m+\ell) A + C = (mA + C) + \ell A\\
    &= (P_m + \ell A) + \Delta_m = (P_m + \ell A + \varepsilon F) + (\Delta_m - \varepsilon F).
    \end{aligned}
    $$
    And since $P_m + \ell A + \varepsilon F$ is nef and $\Delta_m - \varepsilon F$ is effective, we have
    $$
    \Delta_m - \varepsilon F \ge \Delta_{m+\ell}.
    $$
    This implies that $F\not\subset \supp(\Delta_{m+\ell})$, which is a contradiction.
    \end{proof}

From the above lemma, we conclude:
        For $\deg A \gg n \gg m$,
        \begin{align}\label{K3ISOTRIVIALE815}
            \calL \Delta = 0.
        \end{align}
        Thus, $\Delta$ is supported on curves $F$ contained in the fibres of $Z/D$ and satisfying $(K_X+W) \phi_* F = 0$.

    Hence $\Delta$ is uniquely determined by
    \begin{equation}\label{K3ISOTRIVIALE810}
        \begin{aligned}
            (\pi\circ \phi)_* \Delta & = 0,\ (K_X+W)\phi_*\Delta = 0,\ \text{and}                                         \\
            \Delta F                 & = - (K_Z + \Gamma_{s,m})F \hspace{12pt}\text{for all components } F\subset \Delta.
        \end{aligned}
    \end{equation}
    In what follows $\OO_{\lfloor \Delta\rfloor}$ (and similarly for $\OO_{\Delta_p}$) denotes the coherent sheaf
    \[
        \OO_{\lfloor \Delta\rfloor} = \dfrac{\OO_Z}{\OO_Z(-\lfloor \Delta\rfloor)}
    \]
    on $Z$.

    \begin{Lemma}\label{K3ISOTRIVIALLEMDELTA}
        Suppose that $W_p$ is integral and singular for some $p\in D$. Then the fibre $Z_p$ of $Z/D$ can be written as
        \[
            Z_p = Q + R_1 + R_2 + ... + R_\mu
        \]
        for some $\mu\in \ZZ^+$, where $Q\cong C$ and $R_1, \dots, R_\mu \cong \PP^1$ are distinct, and $R_\mu$ is the unique component mapping dominantly to $W_p$. Furthermore, the restriction $\Delta_p = \Delta\cap Z_p$ is a divisor with $\ZZ$-coefficients given by
        \begin{equation}\label{K3ISOTRIVIALE812}
            \Delta_p = (3g + m - 1) (\mu Q + (\mu-1)R_1 + (\mu-2)R_2 + ... + R_{\mu-1})
        \end{equation}
        and it satisfies $\OO_{\Delta_p} \otimes \OO_Z(\calL) \cong \OO_{\Delta_p}$.
    \end{Lemma}
    \begin{proof}
        The strict transform of $C\times\{p\}\subset Y$ in $Z$ must be contracted to $W$ since $W_p$ is singular and hence has smaller geometric genus than $C$. Since we chose $\phi\colon Z\to W$ to be the minimal resolution of $f$,  there exists a unique $(-1)$-curve $R$ in $Z_p$ such that $\psi_* R = 0$ and $\phi_* R = W_p$. In addition, since $W_p$ is integral, $Z_p$ is reduced along $R$. It is not hard to see that
        \[
            Z_p = Q + R_1 + R_2 + ... + R_\mu
        \]
        for some $\mu\in \ZZ^+$, where $Q\cong C$, $R_i \cong \PP^1$,
        \[
            \begin{aligned}
                QR_1     & = R_1R_2 = ... = R_{\mu-1} R_{\mu} = 1,    \\
                \phi_* Q & = \phi_* R_1 = ... = \phi_* R_{\mu-1} = 0,
                \text{ and } \phi_* R_\mu = W_p.
            \end{aligned}
        \]
        One then computes from Lemma \ref{K3ISOTRIVIALLEMNODALSING} that
        \[
            \begin{aligned}
                K_Z Q          & = 2g - 1,\ K_ZR_1 = K_Z R_2 = ... = K_ZR_{\mu-1} = 0,\ K_Z R_\mu = -1                    \\
                Q^2            & = R_\mu^2 = -1,\ R_1^2 = R_2^2 = ... = R_{\mu-1}^2 = -2                                  \\
                \Gamma_{s,m} Q & = g +m, \text{ and } \Gamma_{s,m} R_1 = \Gamma_{s,m} R_2 = ... = \Gamma_{s,m} R_\mu = 0.
            \end{aligned}
        \]
        Since $(K_X+W) \phi_* R_\mu = (K_X + W)W_p = 2g-2 > 0$, it follows from Lemma \ref{K3ISOTRIVIALE815} that $\Delta_p$ is supported on $Q\cup R_1\cup ...\cup R_{\mu-1}$. Using \eqref{K3ISOTRIVIALE810} to compute $\Delta_p$, we obtain exactly \eqref{K3ISOTRIVIALE812}. Because all coefficients in \eqref{K3ISOTRIVIALE812} are integers, we have $\lfloor \Delta_p \rfloor = \Delta_p$. Also notice that since $\phi_* \Delta_p = 0$ and $\calL$ is a pullback, $\calL$ restricts trivially to $\Delta_p$, so $\OO_{\Delta_p} \otimes \OO_Z(\calL) \cong \OO_{\Delta_p}$.
    \end{proof}

    Note that $\Delta$ has simple normal crossing support since it is contained in the fibres of $Z/D$. So by Kawamata--Viehweg vanishing (cf.\ \cite{E-V} or more explicitly \cite[Theorem 9.1.18]{lazarsfeldpos2}), we have
    \[
        \HH^i(\calL - \Gamma_{s,m} - \lfloor \Delta\rfloor) =
        \HH^i(K_Z + P + (\Delta - \lfloor \Delta \rfloor)) = 0
    \]
    for $i\ge 1$. Thus, from the exact sequence
    \[
        \begin{tikzcd}[column sep=1em]
            0 \ar{r} & \HH^i(\calL - \Gamma_{s,m} - \lfloor \Delta\rfloor) \ar{r} & \HH^i(\calL - \Gamma_{s,m}) \ar{r} & \HH^i(\OO_{\lfloor \Delta\rfloor}(\calL - \Gamma_{s,m})) \ar{r} & 0
        \end{tikzcd}
    \]
    we see that
    \[
        \HH^i(\calL - \Gamma_{s,m}) = \begin{cases}
            \HH^1(\OO_{\lfloor \Delta\rfloor}(\calL - \Gamma_{s,m})) & \text{for } i = 1 \\
            0                                                        & \text{for } i=2.
        \end{cases}
    \]
    Thus,
    \[
        h^0(\calL - \Gamma_{s,m}) = \chi(\OO_Z(\calL - \Gamma_{s,m}))
        + h^1(\OO_{\lfloor \Delta\rfloor}(\calL - \Gamma_{s,m}))
    \]
    for $\deg A \gg n\gg m$. As the family is flat, $\chi(\OO_Z(\calL - \Gamma_{s,m}))$ is constant over $s\in U$. Therefore, to prove \eqref{K3ISOTRIVIALE809}, it suffices to prove
    \[
        h^1(\OO_{\lfloor \Delta\rfloor}(\calL - \Gamma_{s,m}))
        \ne h^1(\OO_{\lfloor \Delta\rfloor}(\calL - \Gamma_{t,m}))
    \]
    for some $m\in \ZZ^+$ and $s,t\in U$. For a point $p\in D$, let $Z_p$ be the fibre of $Z/D$ and let $\Delta_p = \Delta\cap Z_p$. Then
    \[
        h^1(\OO_{\lfloor \Delta\rfloor}(\calL - \Gamma_{s,m}))
        = \sum_{p\in D} h^1(\OO_{\lfloor \Delta_p\rfloor}(\calL - \Gamma_{s,m})).
    \]
    Clearly, each $h^1(\OO_{\lfloor \Delta_p\rfloor}(\calL - \Gamma_{s,m}))$ is upper-semicontinuous in $s\in U$. So it suffices to prove
    \begin{equation}\label{K3ISOTRIVIALE811}
        h^1(\OO_{\lfloor \Delta_p\rfloor}(\calL - \Gamma_{s,m}))
        \ne h^1(\OO_{\lfloor \Delta_p\rfloor}(\calL - \Gamma_{t,m}))
    \end{equation}
    for some $p\in D$, $m\in \ZZ^+$ and $s,t\in U$.

    By Lemma \ref{K3ISOTRIVIALLEMDELTA}, if $W_p$ is integral and singular for some $p \in D$, the equation \eqref{K3ISOTRIVIALE811} reduces to
    \begin{equation}\label{K3ISOTRIVIALE811B}
        h^1(\OO_{\Delta_p}(-\Gamma_{s,m})) \ne h^1(\OO_{\Delta_p}(-\Gamma_{t,m})).
    \end{equation}

    \medskip\noindent\textbf{The case $\mu=1$.}
    Let us first prove \eqref{K3ISOTRIVIALE811B} for $\mu=1$. Although this is not logically necessary, we feel that this simple case will help the reader to understand the proof better.

    When $\mu = 1$,
    \[\Delta_p = (3g + m - 1) Q\]
    by \eqref{K3ISOTRIVIALE812}.

    We can compute $h^1(\OO_{nQ}(-\Gamma_{s,m}))$ as a combination of the following lemmas.

    \begin{Lemma}\label{LEM:K3ISOTRIVIALE838}
        There are short exact sequences
        \[
            \begin{tikzcd}[column sep=17pt]
                0 \ar{r} & \OO_C(kq - \sigma)
                \ar{r} & \OO_{(k+1)Q}(-\Gamma_{s,m})
                \ar{r} & \OO_{kQ}(-\Gamma_{s,m}) \ar{r} & 0
            \end{tikzcd}
        \]
        on $Z$ for $k \ge 0$, where
        \[
            \sigma = s_1 + s_2 + ... + s_g + m s_{g+1}.
        \]
    \end{Lemma}
    \begin{proof}
        From the filtration
        \begin{equation}\label{K3ISOTRIVIALE813}
            \OO_Z \supset \OO_Z(-Q) \supset \OO_Z(-2Q) \supset \cdots,
        \end{equation}
        the kernel is the sheaf
        \[
            \dfrac{\OO_Z(-k Q)}{\OO_Z(-(k+1)Q)}\otimes  \OO_Z(-\Gamma_{s,m}).
        \]
        Since the fibre $Z_p = Q+R_1$ has trivial normal bundle in $Z$, we have $\OO_Q(Q+R_1) \cong \OO_Q$, giving $\OO_Q(Q) \cong \OO_Q(-R_1)$. Under the isomorphism $ Q \cong C$, this gives $\OO_Q(-kQ) \cong \OO_C(kq)$ where $q = \pi_C(Q\cap R_1)$, giving the desired short exact sequence.
    \end{proof}

    By iterating Lemma \ref{LEM:K3ISOTRIVIALE838} above, we obtain
    \begin{equation}\label{K3ISOTRIVIALE839}
        \chi(\OO_{kQ}(-\Gamma_{s,m})) = \sum_{j=0}^{k-1} \chi(\OO_C(jq - \sigma))
    \end{equation}
    for $k\ge 0$, which is independent of $s\in U$.

    \begin{Lemma}\label{LEM:K3ISOTRIVIALE840}
        For $s\in U$ general,
        \[
                \quad h^1(\OO_{(3g+m-1)Q}(-\Gamma_{s,m})) = - \chi(\OO_{(2g+m)Q}(-\Gamma_{s,m}))
        \]
        which is in particular independent of $s\in U$.
    \end{Lemma}
    \begin{proof}
        By Riemann-Roch, a general line bundle of degree $d$ on a curve of genus $g$ has exactly $\max(0, d - g + 1)$ global sections. Thus,
        \[
            h^0(\OO_C(kq - \sigma))
            = \max(0, k -2g -m +1)
        \]
        for $s\in U$ general. In other words,
        \[
            \text{either }
            h^0(\OO_C(kq - \sigma))
            = 0 \text{ or }
            h^1(\OO_C(kq - \sigma))
            = 0
        \]
        for each $0\le k\le 3g+m-2$ and $s\in U$ general.
        From Lemma \ref{LEM:K3ISOTRIVIALE838},
        \[
            \begin{aligned}
                h^0(\OO_{Q}(-\Gamma_{s,m}))
                                                  & = h^0(\OO_{2Q}(-\Gamma_{s,m}))
                \\ & = \ldots \\ & = h^0(\OO_{(2g + m)Q}(-\Gamma_{s,m})) \\ & = 0                           \\
                h^1(\OO_{(2g+m)Q}(-\Gamma_{s,m})) & =
                h^1(\OO_{(2g+m+1)Q}(-\Gamma_{s,m})) \\ & = \ldots                                 \\
                                                  & = h^1(\OO_{(3g+m-1)Q}(-\Gamma_{s,m}))
            \end{aligned}
        \]
        and hence
        \[
            \begin{aligned}
                 & \quad h^1(\OO_{(3g+m-1)Q}(-\Gamma_{s,m})) = h^1(\OO_{(2g + m)Q}(-\Gamma_{s,m})) \\
                 & = h^0(\OO_{(2g+m)Q}(-\Gamma_{s,m})) - \chi(\OO_{(2g+m)Q}(-\Gamma_{s,m}))        \\
                 & = - \chi(\OO_{(2g+m)Q}(-\Gamma_{s,m}))
            \end{aligned}
        \]
        so that the equality in Lemma \ref{LEM:K3ISOTRIVIALE840} holds.
    \end{proof}

    \begin{Lemma}\label{LEM:K3ISOTRIVIALE841}
        For arbitrary $s\in U$, we have
        \[
            \begin{aligned}
                h^1(\OO_{(3g+m-1)Q}(-\Gamma_{s,m})) &\ge h^0(\OO_C((2g+m-1)q - \sigma)) \\
                &\quad - \chi(\OO_{(2g+m)Q}(-\Gamma_{s,m})).
            \end{aligned}
        \]
    \end{Lemma}
    \begin{proof}
        For arbitrary $s\in U$, we have from Lemma \ref{LEM:K3ISOTRIVIALE838} that
        \[
            h^0(\OO_C((2g+m-1)q - \sigma)) \le h^0(\OO_{(2g + m)Q}(-\Gamma_{s,m}))
        \]
        and
        \begin{align*}
            h^1(\OO_{(2g+m)Q}(-\Gamma_{s,m})) & \le
            h^1(\OO_{(2g+m+1)Q}(-\Gamma_{s,m})) \le \ldots  \\
            & \le h^1(\OO_{(3g+m-1)Q}(-\Gamma_{s,m})).
        \end{align*}
       In particular,
       \[
            \begin{aligned}
                 & \quad h^1(\OO_{(3g+m-1)Q}(-\Gamma_{s,m})) \ge h^1(\OO_{(2g + m)Q}(-\Gamma_{s,m})) \\
                 & = h^0(\OO_{(2g+m)Q}(-\Gamma_{s,m})) - \chi(\OO_{(2g+m)Q}(-\Gamma_{s,m}))          \\
                 & \ge h^0(\OO_C((2g+m-1)q - \sigma)) - \chi(\OO_{(2g+m)Q}(-\Gamma_{s,m})).
            \end{aligned}
        \]
        and hence the inequality in Lemma \ref{LEM:K3ISOTRIVIALE841} follows.
    \end{proof}
    Comparing Lemma \ref{LEM:K3ISOTRIVIALE840} and Lemma \ref{LEM:K3ISOTRIVIALE841}, we see that \eqref{K3ISOTRIVIALE811} holds as long as we can find $m\in \ZZ^+$ and $s\in U$ satisfying the following Lemma.

    \begin{Lemma}\label{LEM:K3ISOTRIVIALE842}
        There exist $m\in \ZZ^+$ and $s\in U$ such that
        \[
            h^0(\OO_C((2g+m-1)q - \sigma)) > 0.
        \]
    \end{Lemma}
    \begin{proof}
        Note that
        \[
            (2g+m-1)q - \sigma = ((2g-1)q - s_1 - s_2 - ... - s_g) + m(q-s_{g+1}).
        \]
        For $(s_1,s_2,...,s_g)\in C^g$ very general, $(2g-1)q - \sum_{i=1}^g s_i$ is a very general divisor on $C$ of degree $g-1$.
        So by \eqref{K3ISOTRIVIALTHMGLOBALE035} in Proposition \ref{K3ISOTRIVIALPROPJACOBIAN}, for a fixed very general $(s_1,s_2,...,s_g)\in C^g$,
        \[
            \lim_{m\to\infty} \Big |\big\{r\in C: h^0(\OO_C((2g+m-1)q - \sum_{i=1}^g s_i - mr)) > 0 \big\}\Big| = \infty.
        \]
        Let us fix a very general $(s_1,s_2,...,s_g)\in C^g$ with the above property and let $U_{(s_1,s_2,...,s_g)}$ be the fibre of $U$ over $(s_1,s_2,...,s_g)\in C^g$ under the projection $U\to C^g$ sending $(s_1,s_2,...,s_g,s_{g+1})$ to $(s_1,s_2,...,s_g)$. Since $U$ is a Zariski open set, $U_{(s_1,s_2,...,s_g)}$ is a Zariski open set in $C$. So
        \[
            U_{(s_1,s_2,...,s_g)}\cap \big\{r\in C: h^0(\OO_C((2g+m-1)q - \sum_{i=1}^g s_i - mr)) > 0 \big\} \ne \emptyset
        \]
        for $m\gg 0$. That is, there exists $(s_1,s_2,...,s_g,s_{g+1})\in U$ such that $h^0(\OO_C((2g+m-1)q - \sigma)) > 0$ for some $m\gg 0$, proving the claim.
    \end{proof}

    By Lemma \ref{LEM:K3ISOTRIVIALE842}, this proves our overall claim \eqref{K3ISOTRIVIALE811} for $\mu = 1$.

   \medskip\noindent\textbf{The general $\mu$ case.}
   For arbitrary $\mu$, we use the following filtration for $\Delta_p$ given by \eqref{K3ISOTRIVIALE812}:
    \[
        \begin{aligned}
            \OO_Z = \OO_Z(-G_0) & \supset \OO_Z(-G_1)\supset ... \supset \OO_Z(-G_{\mu})                        \\
                                & \supset \OO_Z (-G_{\mu+1}) \supset ... \supset \OO_Z(-G_{2\mu}) \supset ...   \\
                                & \supset \OO_Z(-G_{(3g+m-2)\mu + 1})\supset ...\supset \OO_Z(-G_{(3g+m-1)\mu})
        \end{aligned}
    \]
    where $G_0 = 0$ and
    \[
        G_{n\mu+i} - G_{n\mu + i - 1} = Q + R_1 + R_2 + ... + R_{\mu - i}
    \]
    for $0\le n \le 3g+m-2$ and $1\le i \le \mu$. From this filtration, we have short exact sequences:
    \[
        \begin{tikzcd}[column sep=small]
            & \OO_{G_{k+1}- G_k}(-\Gamma_{s, m}-G_k)\ar[equal]{d}\\
            0 \ar{r} & \dfrac{\OO_Z(-\Gamma_{s,m}-G_{k})}{\OO_Z(-\Gamma_{s,m}-G_{k+1})} \ar{r} & \OO_{G_{k+1}}(-\Gamma_{s,m}) \ar{r} & \OO_{G_k}(-\Gamma_{s,m}) \ar{r} & 0
        \end{tikzcd}
    \]
    for all $k\ge 0$.
    For $k = n\mu+i - 1$, we have
    \[
        \OO_{G_{k+1}- G_k}(-\Gamma_{s,m} - G_k) = \OO_{Q+R_1+...+R_{\mu - i}}(-\Gamma_{s,m} - G_k)
    \]
    with
    \[
        \begin{aligned}
            \OO_Q (-\Gamma_{s,m} - G_k)     & \cong \OO_C(n q - \sigma)
            \hspace{12pt}\text{and}                                     \\
            \OO_{R_j} (-\Gamma_{s,m} - G_k) & \cong \OO_{R_j}
            \hspace{12pt}\text{for } j=1,2,...,\mu-i
        \end{aligned}
    \]
    where $q = \pi_C(R_1\cup R_2\cup ...\cup R_\mu)$. Therefore,
    \[
        h^\bullet(\OO_{G_{k+1}- G_k}(-\Gamma_{s,m} - G_k)) = h^\bullet(\OO_C(nq - \sigma))
    \]
    for $k = n\mu+i - 1$, $0\le n \le 3g + m -2$ and $1\le i \le \mu$. Then the rest of the argument is the same as that for $\mu=1$.
\end{proof}


\begin{thebibliography}{McK93}

\bibitem[Bak25]{bakker}
Benjamin Bakker.
\newblock A short proof of a conjecture of {Matsushita}.
\newblock {\em Advances in Mathematics}, 482, 2025.

\bibitem[Bea90]{beauville1990}
Arnaud Beauville.
\newblock Sur les hypersurfaces dont les sections hyperplanes sont \`a module
  constant.
\newblock In {\em The Grothendieck Festschrift, Volume I}, pages 121--133.
  Birkh\"auser, 1990.

\bibitem[Bea25]{beauville2025maximal}
Arnaud Beauville.
\newblock Maximal variation of linear systems.
\newblock \url{https://arxiv.org/abs/2511.22329}, 2025.

\bibitem[BP26]{bricallipirola2026maximal}
Davide Bricalli and Gian~Pietro Pirola.
\newblock On the maximal variation problem and lefschetz pencils.
\newblock \url{https://arxiv.org/abs/2606.27893}, 2026.

\bibitem[CDS20]{CDS}
Ciro Ciliberto, Thomas Dedieu, and Edoardo Sernesi.
\newblock Wahl maps and extensions of canonical curves and {$K3$} surfaces.
\newblock {\em J. Reine Angew. Math.}, 761:219--245, 2020.

\bibitem[CFZ15]{CFZ}
Ciro Ciliberto, Flaminio Flamini, and Mikhail Zaidenberg.
\newblock Genera of curves on a very general surface in {$\mathbb{P}^3$}.
\newblock {\em Int. Math. Res. Not. IMRN}, (22):12177--12205, 2015.

\bibitem[CG22]{maxmoduli}
Xi~Chen and Frank Gounelas.
\newblock Curves of maximal moduli on {K3} surfaces.
\newblock {\em Forum of Mathematics, Sigma}, 10:e36, 2022.

\bibitem[CGL22]{regenerationinfinite}
Xi~Chen, Frank Gounelas, and Christian Liedtke.
\newblock Curves on {K}3 surfaces.
\newblock {\em Duke Math. J.}, 171(16):3283--3362, 2022.

\bibitem[DH22]{DuttaHuybrechts}
Yajnaseni Dutta and Daniel Huybrechts.
\newblock Maximal variation of curves on {K3} surfaces.
\newblock {\em Tunisian Journal of Mathematics}, 4(3):443--464, 2022.

\bibitem[DS23]{dedieu-sernesi}
Thomas Dedieu and Edoardo Sernesi.
\newblock Deformations and extensions of {G}orenstein weighted projective
  spaces.
\newblock In {\em The art of doing algebraic geometry}, Trends Math., pages
  119--143. Birkh\"auser/Springer, Cham, 2023.

\bibitem[EV92]{E-V}
H\'el\`ene Esnault and Eckart Viehweg.
\newblock {\em Lectures on Vanishing Theorems}.
\newblock DMV Seminar Vol 20. Springer-Verlag, Berlin-New York, 1992.

\bibitem[Gar17]{garbagnati}
Alice Garbagnati.
\newblock On {K}3 surface quotients of {K}3 or {A}belian surfaces.
\newblock {\em Canad. J. Math.}, 69(2):338--372, 2017.

\bibitem[Har77]{Hartshorne}
Robin Hartshorne.
\newblock {\em Algebraic geometry}.
\newblock Springer-Verlag, New York, 1977.
\newblock Graduate Texts in Mathematics, No. 52.

\bibitem[Hei06]{ghein}
Georg Hein.
\newblock Restriction of stable rank two vector bundles in arbitrary
  characteristic.
\newblock {\em Communications in Algebra}, 34(7):2319--2335, 2006.
\newblock Also preprint arXiv:math/9904102.

\bibitem[HM98]{harrismorrison}
Joe Harris and Ian Morrison.
\newblock {\em Moduli of curves}, volume 187 of {\em Graduate Texts in
  Mathematics}.
\newblock Springer-Verlag, New York, 1998.

\bibitem[Keu16]{keum2016}
JongHae Keum.
\newblock Orders of automorphisms of {K}3 surfaces.
\newblock {\em Advances in Mathematics}, 303:39--87, 2016.

\bibitem[Knu01]{knutsenk-th}
Andreas~Leopold Knutsen.
\newblock On {$k$}th-order embeddings of {$K3$} surfaces and {E}nriques
  surfaces.
\newblock {\em Manuscripta Math.}, 104(2):211--237, 2001.

\bibitem[Knu20]{knutsen-globalsections}
Andreas~Leopold Knutsen.
\newblock Global sections of twisted normal bundles of {$K3$} surfaces and
  their hyperplane sections.
\newblock {\em Atti Accad. Naz. Lincei Rend. Lincei Mat. Appl.}, 31(1):57--79,
  2020.

\bibitem[Kon92]{kondo1992}
Shigeyuki Kond\=o.
\newblock Automorphisms of algebraic {$K3$} surfaces which act trivially on
  {P}icard groups.
\newblock {\em J. Math. Soc. Japan}, 44(1):75--98, 1992.

\bibitem[Laz04]{lazarsfeldpos2}
Robert Lazarsfeld.
\newblock {\em Positivity in algebraic geometry. {II}}, volume~49 of {\em
  Ergebnisse der Mathematik und ihrer Grenzgebiete. 3. Folge. A Series of
  Modern Surveys in Mathematics [Results in Mathematics and Related Areas. 3rd
  Series. A Series of Modern Surveys in Mathematics]}.
\newblock Springer-Verlag, Berlin, 2004.
\newblock Positivity for vector bundles, and multiplier ideals.

\bibitem[McK93]{mckernan1993}
James McKernan.
\newblock Varieties with isomorphic or birational hyperplane sections.
\newblock {\em International Journal of Mathematics}, 4(1):113--125, 1993.

\bibitem[Moo17]{moonen}
Ben Moonen.
\newblock {The Deligne-Mostow List and Special Families of Surfaces}.
\newblock {\em International Mathematics Research Notices},
  2018(18):5823--5855, 03 2017.

\bibitem[Par88]{paranjape}
Kapil Paranjape.
\newblock Abelian varieties associated to certain {$K3$} surfaces.
\newblock {\em Compositio Math.}, 68(1):11--22, 1988.

\bibitem[Par94]{pardini1994}
Rita Pardini.
\newblock Some remarks on varieties with projectively isomorphic hyperplane
  sections.
\newblock {\em Geometriae Dedicata}, 52(1):15--32, 1994.

\bibitem[Sch96]{Schoen}
Chad Schoen.
\newblock Varieties dominated by product varieties.
\newblock {\em Internat. J. Math.}, 7(4):541--571, 1996.

\bibitem[Tot20]{totaro}
Burt Totaro.
\newblock Bott vanishing for algebraic surfaces.
\newblock {\em Transactions of the American Mathematical Society},
  373(5):3609--3626, 2020.

\bibitem[Xu94]{xu1994}
Geng Xu.
\newblock Subvarieties of general hypersurfaces in projective space.
\newblock {\em J. Differential Geom.}, 39(1):139--172, 1994.

\end{thebibliography}
\end{document}